\documentclass[11pt]{amsart}

\usepackage[a4paper,left=30mm,right=30mm,top=30mm,bottom=30mm]{geometry}
\usepackage{amssymb,mathrsfs}

\usepackage{xcolor}
\definecolor{linkblue}{RGB}{38,68,105}
\definecolor{citegreen}{RGB}{45,88,68}

\usepackage[colorlinks=true,linkcolor=linkblue,citecolor=citegreen,urlcolor=linkblue,pdfborder={0 0 0}]{hyperref}

\newcommand{\C}{\mathbb C}
\newcommand{\R}{\mathbb R}
\newcommand{\Q}{\mathbb Q}
\newcommand{\Z}{\mathbb Z}
\newcommand{\N}{\mathbb N}
\newcommand{\A}{\mathbb A}
\newcommand{\PP}{\mathbb P}
\newcommand{\OO}{\mathcal O}
\newcommand{\MM}{\mathcal M}
\newcommand{\RR}{\mathscr R}

\DeclareMathOperator{\gr}{gr}
\DeclareMathOperator{\Rees}{Rees}
\DeclareMathOperator{\Spec}{Spec}
\DeclareMathOperator{\Supp}{Supp}
\DeclareMathOperator{\Ann}{Ann}
\DeclareMathOperator{\Tor}{Tor}
\DeclareMathOperator{\im}{im}
\DeclareMathOperator{\Frac}{Frac}
\DeclareMathOperator{\Ch}{Ch}
\DeclareMathOperator{\CC}{CC}
\DeclareMathOperator{\conv}{conv}

\DeclareMathOperator{\spc}{sp}
\newcommand{\ini}{\operatorname{in}}

\newcommand{\DD}{\mathcal D}
\newcommand{\RH}{\mathbf R\mathcal Hom}
\newcommand{\one}{\mathbf 1}
\DeclareMathOperator{\Ext}{Ext}
\DeclareMathOperator{\qdeg}{qdeg}
\DeclareMathOperator{\rank}{rank}
\DeclareMathOperator{\depth}{depth}

\newtheorem{theorem}{Theorem}[section]
\newtheorem{proposition}[theorem]{Proposition}
\newtheorem{lemma}[theorem]{Lemma}
\newtheorem{corollary}[theorem]{Corollary}
\theoremstyle{definition}
\newtheorem{definition}[theorem]{Definition}
\theoremstyle{remark}
\newtheorem{remark}[theorem]{Remark}
\newtheorem{example}[theorem]{Example}
\numberwithin{equation}{section}

\title[Characteristic cycles and Gevrey irregularity]{Characteristic cycles and Gevrey irregularity in hypergeometric families}

\author{Sheng Tan}
\address{School of Mathematical Sciences, Capital Normal University, Beijing 100048, China} 
\email{tansheng2018@outlook.com}



\thanks{This work is supported by NSFC (Grant No.~12601116) and BJNSF (Grant No.~1264052).}

\date{}

\begin{document}

\begin{abstract}
We show that the characteristic cycle multiplicities of $A$-hypergeometric systems are upper semicontinuous in the parameter for every rational projective weight, thereby proving a conjecture of Schulze and Walther [Duke Math. J. \textbf{142} (2008), 465--509]. The proof uses a specialization formula for filtered families satisfying a fixed product support condition, expressing the difference between the special and generic characteristic cycles as the characteristic cycle of the first Tor module. Combined with the Gevrey index theorem, this yields upper semicontinuity of generic Gevrey irregularity and of each of its graded dimensions. We also construct non-Cohen--Macaulay semigroup rings whose Gevrey dimensions are constant in the parameter despite the presence of a genuine irregular slope, showing that the Gevrey jump locus need not coincide with the full exceptional arrangement defined by Ext.
\end{abstract}

\maketitle

\tableofcontents


\section{Introduction}\label{sec:intro}

The characteristic cycle of a system of differential equations records its singular directions together with their multiplicities, which are related to local properties of its solutions through microlocal index theory \cite{Kas85,KS90}. For $A$-hypergeometric systems, its support is independent of the parameter, while its coefficients may change. In this paper, we study these changes through a formula for specialization of characteristic cycles and apply the same formula to the Gevrey filtration of irregularity.

Let $A=(a_1,\ldots,a_n)\in\Z^{d\times n}$ be an integer matrix of rank $d$ with nonzero columns. Write $\N A=\sum_{j=1}^n\N a_j$, $\Z A=\sum_{j=1}^n\Z a_j$, and $\R_{\ge0}A=\sum_{j=1}^n\R_{\ge0}a_j$. We assume that $\N A$ is pointed, that is, $\N A\cap(-\N A)=\{0\}$. Put
\[
D=\C[x_1,\ldots,x_n]\langle\partial_1,\ldots,\partial_n\rangle,
\quad P=\C[\partial_1,\ldots,\partial_n],
\quad S_A=\C[\N A]=P/I_A,
\]
where $I_A$ is generated by the binomials $\partial^p-\partial^q$ with $Ap=Aq$. For $\beta\in\C^d$, set
\begin{equation}\label{eq:intro-gkz}
M_A(\beta)=D\big/\bigl(DI_A+\textstyle\sum_{i=1}^dD(E_i-\beta_i)\bigr),
\quad E_i=\sum_{j=1}^n a_{ij}x_j\partial_j.
\end{equation}
The theory originates in the work of Gel'fand, Graev, and Zelevinsky \cite{GGZ87}; these $A$-hypergeometric systems, commonly called GKZ systems, were introduced and developed by Gel'fand, Kapranov, and Zelevinsky \cite{GZK89,GKZ90}. Their toric geometry and discriminants are treated in \cite{GKZ94}, and their initial ideals, subdivisions, and Gr\"obner deformations in \cite{Stu96,SST00}; see \cite{RSSW21} for a survey. The modules $M_A(\beta)$ are holonomic and form a holonomic family \cite{MMW05,BGM15}.

A rational weight $L=(u,v)\in\Q^n\times\Q^n$ is projective if $u_j+v_j=c>0$ for every $j$. We use $L$ to define the filtration on $D$ by assigning weight $u_j$ to $x_j$ and weight $v_j$ to $\partial_j$. For a projective weight, the associated graded ring is $\gr^L D=W:=\C[x,\xi]$. 
The characteristic components correspond to the elements of the $(A,L)$-umbrella $\Phi_A^L$ as follows:
\begin{equation}\label{eq:intro-cycle}
\CC^L(M_A(\beta))=\sum_{\tau\in\Phi_A^L}\mu_{A,0}^{L,\tau}(\beta)[\Lambda_\tau],
\quad \Lambda_\tau=\overline{C_A^\tau}.
\end{equation}
Here $C_A^\tau$ is the conormal space associated with a torus orbit in the convention of \cite{SW08}; see Section~\ref{sec:gkz}. Although the varieties $\Lambda_\tau$ are irreducible of dimension $n$ and independent of $\beta$, their multiplicities may vary with $\beta$.

For the order filtration $F=(0,\mathbf1)$, the coefficient of the zero section is the holonomic rank, whose generic value is the normalized volume \cite{Ado94}. Early work on exceptional parameters and rank jumps includes \cite{ST98,CDD99,Mat01,Mat03}, with further studies of simplicial configurations and toric threefolds in \cite{Sai02,MM06,Oku06}. 
Matusevich, Miller, and Walther proved that the rank is upper semicontinuous and characterized exceptional parameters in terms of lower local cohomology \cite{MMW05}. Subsequent work developed the theory of holonomic families and ranking toric modules \cite{BGM15,Ber11}, and obtained unbounded rank gaps and bounds relative to volume \cite{MW07,BFF22,BFF24}. The variation of series solutions on parameter strata is studied in \cite{BFM16,BFM18}.

For general projective weights, Schulze and Walther computed the generic multiplicities and proved that they are lower bounds at every parameter \cite[Theorems~4.21 and~4.28]{SW08}. They ask in \cite[Conjecture~4.27]{SW08} whether each coefficient $\mu_{A,0}^{L,\tau}(\beta)$ in \eqref{eq:intro-cycle} is upper semicontinuous. Berkesch and Fern\'andez-Fern\'andez computed the cycles at arbitrary parameters and proved the conjecture for convex filtrations at the zero section, and under $\tau$-convexity and a pyramid condition for other components \cite[Theorems~4.1 and~6.1, Corollary~6.6]{BFF20}. 
The following theorem proves \cite[Conjecture~4.27]{SW08} in its full generality.

\begin{theorem}\label{thm:gkz}
Let $A\in\Z^{d\times n}$ be an integer matrix of rank $d$ with nonzero columns, and assume that $\N A$ is pointed. Let $L=(u,v)\in\Q^n\times\Q^n$ be a projective weight. For every $\tau\in\Phi_A^L$ and every integer $r$, the set
\[
\{\beta\in\C^d:\mu_{A,0}^{L,\tau}(\beta)\ge r\}
\]
is Zariski closed.
\end{theorem}

Let $B$ be a finitely generated commutative $\C$-algebra, and put $D_B=D\otimes_\C B$, with the elements of $B$ in filtration degree zero.
Suppose that a finite $D_B$-module $M$ has a good $L$-filtration satisfying the following \emph{product support condition}:
\begin{equation}\label{eq:intro-support}
\Supp\gr^L M\subseteq\bigcup_{i=1}^q\Lambda_i\times\Spec B,
\end{equation}
where the $\Lambda_i\subseteq\Spec W$ are fixed, distinct, irreducible subvarieties of dimension $n$. Write $M_b=M\otimes_B\kappa(b)$ for a closed fiber, where $\kappa(b)$ is the residue field of $b$.

\begin{theorem}\label{thm:general}
Assume \eqref{eq:intro-support}. There is a finite stratification of $\Spec B$ by locally closed algebraic subsets on which the multiplicities of $\CC^L(M_b)$ along $\Lambda_1,\ldots,\Lambda_q$ are simultaneously constant. Moreover, each multiplicity is Zariski upper semicontinuous on the closed points of $\Spec B$.
\end{theorem}

More precisely, let $R$ be the local ring of a smooth complex algebraic curve at a closed point, with uniformizer $t$, residue field $k=\C$, and fraction field $K$. Under the corresponding support condition, we prove
\begin{equation}\label{eq:intro-tor}
\CC^L(M_R/tM_R)-\spc_\Lambda\CC^L(M_R\otimes_R K)
=\CC^L\bigl(\Tor_1^R(M_R,k)\bigr).
\end{equation}
Here $\spc_\Lambda$ is the homomorphism sending $[\Lambda_i\times\Spec K]$ to $[\Lambda_i]$ for every $i$. For the order filtration, equality of the two cycles is equivalent to $R$-flatness (Corollary~\ref{cor:flatness}). The proof compares the local parameter $t$ with the Rees parameter $h$: the failure of associated graded modules to commute with specialization cancels in a local Euler characteristic. 
The effective term in \eqref{eq:intro-tor} arises from $t$-torsion. We prove constructibility using generic freeness and a finite prime filtration.
The relevant background includes filtered $D$-modules \cite{Laum83,Bjo93}, specialization and multifiltrations \cite{Sab87,CG07}, and relative filtrations \cite{Mai23}. We use the comparison between kernels and cokernels underlying the restriction formula for the order filtration in \cite[Proposition~2.7]{Wu22}, now for arbitrary rational projective weights.

We next consider Gevrey irregularity along $Y=\{x_n=0\}\subset X=\A^n$. For a holonomic $D$-module $N$ and $s>1$, let $\OO_{\widehat{X|Y}}(s)$ denote the sheaf of Gevrey series of order at most $s$, and set
\begin{equation}\label{eq:irregularity}
I_s(N)=\RH_{\DD_{X^{\mathrm{an}}}}\bigl(N^{\mathrm{an}},
\OO_{\widehat{X|Y}}(s)/\OO_{X|Y}\bigr),
\quad d_s(N)=\rank_Y\mathcal H^0 I_s(N).
\end{equation}
Thus $d_s(N)$ is the dimension at a general point of $Y$ of Gevrey solutions modulo convergent solutions. 
The exact Gevrey filtration, its positivity, and the comparison between algebraic and analytic slopes are established in \cite{Meb90,LM99}, using the microlocal theory of slopes \cite{Lau87}.
For hypergeometric systems, explicit Gevrey and Nilsson solutions are studied in \cite{FF10,FC11,DMM12,FF20}, and integral representations and summation in \cite{CG15,CFG21,ET15,CFKT15}.

Let $\Sigma(A,Y)\subset\Q_{>1}$ be the finite set of umbrella walls for $L(r)=F+(r-1)V_Y$, where $V_Y$ is the filtration along $Y$. 
The filtration $V_Y$ assigns weights $-1$ to $x_n$, $1$ to $\partial_n$, and $0$ to the other generators. An umbrella wall is a value of $r$ at which the $(A,L(r))$-umbrella changes.
For $\rho>1$, choose non-wall levels $1<a<\rho<b$ such that $[a,b]$ contains no member of $\Sigma(A,Y)$ other than possibly $\rho$. Let $e_\rho(N)$ be the generic dimension of the solution sheaf with values in $\OO_{\widehat{X|Y}}(b)/\OO_{\widehat{X|Y}}(a)$. For the modules below it equals $d_b(N)-d_a(N)$ and is independent of sufficiently close choices of $a,b$.

\begin{theorem}\label{thm:gkz-app}
Under the assumptions on $A$ in Theorem~\ref{thm:gkz}, the functions
\[
\beta\longmapsto d_s(M_A(\beta)), \quad \beta\longmapsto e_\rho(M_A(\beta))
\]
are Zariski upper semicontinuous for every real $s>1$ and $\rho>1$. Moreover, the second function is zero if $\rho\notin\Sigma(A,Y)$, and the levels used to define $e_\rho$ may be chosen independently of $\beta$.
\end{theorem}

A subset of columns of $A$ is called $F$-homogeneous if its columns lie in a common affine hyperplane not containing the origin, as in \cite{BFF20}.
This proves the upper semicontinuity assertion in \cite[Conjecture~7.10]{BFF20} and removes the assumption in \cite[Lemma~7.8]{BFF20} that exactly one facet is not $F$-homogeneous. The index formula expresses $d_s$ as a combination of characteristic multiplicities with both signs, so upper semicontinuity of the individual coefficients is not sufficient. Applying \eqref{eq:intro-tor} at the required weights gives, for the same module $T=\Tor_1^R(M_R,k)$,
\begin{equation}\label{eq:intro-gevrey-tor}
d_s(M_0)-d_s(M_{\mathrm{gen}})=d_s(T),
\quad e_\rho(M_0)-e_\rho(M_{\mathrm{gen}})=e_\rho(T).
\end{equation}
Here $M_{\mathrm{gen}}$ is a general complex fiber of the parameter curve. The right-hand sides are nonnegative by the Gevrey filtration of $T$. We prove in Section~\ref{sec:gevrey} that the slopes of the modules involved lie in a common finite set.

The remaining part of \cite[Conjecture~7.10]{BFF20} proposes a description of Gevrey jumps in terms of Ext, motivated by the relation between Euler--Koszul homology and local cohomology \cite{MMW05,Ber11} and by graded local duality \cite{BH98,MS05}. Use the positive grading $\deg\partial_j=a_j$ and put $\mathfrak m=(\partial_1,\ldots,\partial_n)\subset P$. Let $S_A^{\{n\}}=\C[(\N A+\Z a_n)\cap\R_{\ge0}A]$ be the partial saturation and let $\mathcal Z(N)$ be the exceptional arrangement defined by lower local cohomology in \eqref{eq:exceptional}. We prove
\begin{equation}\label{eq:intro-ext-inclusion}
\{\beta:d_s(M_A(\beta))>d_s(A)\}\subseteq\mathcal Z(S_A^{\{n\}}),
\end{equation}
where $d_s(A)$ is the generic parameter value. The analogous inclusion holds for $e_\rho$. Hence Cohen--Macaulayness of $S_A^{\{n\}}$ implies constancy of these dimensions. The converse fails: for every $K\ge4$, we construct in Theorem~\ref{thm:counterexample} a rank-three matrix $A_K$ with eight columns such that $S_{A_K}^{\{8\}}=S_{A_K}$ has depth two, its exceptional arrangement consists of $K$ affine lines, and
\[
d_s(A_K,\beta)=
\begin{cases}
0,&1<s<K/3,\\
36,&s\ge K/3
\end{cases}
\quad\text{for every }\beta\in\C^3.
\]
Its only slope is $K/3>1$. The homological contribution to its rank jumps consists of regular Kummer modules, so the Gevrey jump locus is empty even above that slope. Thus the full homological obstruction can be nonzero while its irregular contribution vanishes.

The paper is organized as follows. Section~\ref{sec:filtered} introduces the filtered algebraic framework and the relative support condition. We prove the specialization formula over a parameter curve and derive the flatness criterion in Section~\ref{sec:dvr}. The constructibility and upper semicontinuity of the multiplicities in the characteristic cycles are established in Section~\ref{sec:constructibility}. We then apply these results to $A$-hypergeometric systems in Section~\ref{sec:gkz}, where we also recall the Euler--Koszul constructions needed later. We turn to Gevrey irregularity and its graded dimensions in Section~\ref{sec:gevrey}, where we prove the corresponding specialization formulas. Finally, Section~\ref{sec:ext} relates the corresponding jump loci to local cohomology and provides counterexamples to the proposed description in terms of Ext.


\section{Projective weights and filtered families}\label{sec:filtered}

We work over $\C$ and use left modules. We use the standard conventions for holonomic modules and characteristic cycles in \cite{Bjo93,HTT08,KS90}, and the description of good filtrations in terms of Rees modules in \cite{Laum83}. The same algebraic constructions will also be used after extending scalars to the function field of a parameter curve. A finite module means a finitely generated module over the indicated ring. All filtrations are increasing and exhaustive. For a module $V$ and a central element $z$, the notation $0:_Vz$ means the kernel of multiplication by $z$, and $H^0_{(z)}V=\bigcup_{a\ge1}(0:_Vz^a)$ denotes the submodule of $z$-power torsion.

Let $B$ be a finitely generated commutative $\C$-algebra, and put $D_B=D\otimes_\C B$, with the elements of $B$ in filtration degree zero. We first recall the filtrations determined by projective weights.

\begin{definition}\label{def:projective-filtration}
A rational weight $L=(u,v)\in\Q^n\times\Q^n$ is \emph{projective} if $u_j+v_j=c>0$ for every $j$. The associated filtration on $D_B$ is
\begin{equation}\label{eq:filtration-D}
F_k^LD_B=\sum_{u\cdot\alpha+v\cdot\gamma\le k}B x^\alpha\partial^\gamma.
\end{equation}
\end{definition}

Fix such a weight $L$. Multiplying $L$ by a positive integer changes the indexing of the filtration but not the associated graded supports or multiplicities. We therefore take $u,v\in\Z^n$ and $u_j+v_j=c\in\Z_{>0}$. The inequality $c>0$ implies that $\gr^LD_B=W_B:=B[x,\xi]$, where $\xi_j$ denotes the principal symbol of $\partial_j$ for $1\le j\le n$. We usually omit the superscript $L$ from a chosen good filtration on a module, and write $\gr_FM$ when it is useful to distinguish that filtration from the weight itself.

We use Rees modules to describe good filtrations and their behavior under specialization.

\begin{definition}\label{def:rees-good-filtration}
The \emph{Rees algebra} of $D_B$ is
\[
\RR_B=\Rees_LD_B=\bigoplus_{k\in\Z}F_k^LD_B\,h^k\subseteq D_B[h,h^{-1}].
\]
For a finite $D_B$-module $M$ with a filtration $F_\bullet M$ compatible with $F_\bullet^LD_B$, its \emph{Rees module} is
\[
\mathscr M=\Rees_FM=\bigoplus_{k\in\Z}F_kM\,h^k\subseteq M[h,h^{-1}].
\]
The filtration is \emph{good} if $\mathscr M$ is finite over $\RR_B$.
\end{definition}

The Rees algebra has the presentation
\begin{equation}\label{eq:rees-pbw}
\RR_B=\frac{B\langle X_1,\ldots,X_n,Y_1,\ldots,Y_n,h\rangle}
{([Y_i,X_j]-\delta_{ij}h^c,\ [X_i,X_j],\ [Y_i,Y_j],\ [h,X_i],\ [h,Y_i])},
\end{equation}
where $X_i=x_ih^{u_i}$ and $Y_i=\partial_ih^{v_i}$. In particular,
\begin{equation}\label{eq:rees-fibers}
\RR_B/h\RR_B=W_B,
\quad \RR_B[h^{-1}]=D_B[h,h^{-1}].
\end{equation}
Indeed, every ordered monomial $x^\alpha\partial^\gamma h^k$ in $\RR_B$ can be written as
\[
x^\alpha\partial^\gamma h^k
=X^\alpha Y^\gamma h^{k-u\cdot\alpha-v\cdot\gamma},
\]
where the exponent on the right is nonnegative. This shows that $X_1,\ldots,X_n,Y_1,\ldots,Y_n,h$ generate $\RR_B$. The defining relations follow from the Weyl algebra relations, and the PBW basis of $D_B$ shows that the ordered monomials $X^\alpha Y^\gamma h^m$, with $m\ge0$, are linearly independent over $B$. This proves \eqref{eq:rees-pbw}, including when some $u_i$ or $v_i$ is negative. As an iterated Ore extension of $B[h]$, the algebra $\RR_B$ is Noetherian; cf.~\cite[Lemma~4.1]{SW08}.

A filtration is good precisely when it is generated by finitely many elements of $M$ placed in fixed degrees. For its Rees module we have
\[
\mathscr M/h\mathscr M=\gr_FM,
\quad \mathscr M[h^{-1}]=M[h,h^{-1}],
\]
and $h$ acts injectively on $\mathscr M$. The Noetherianity of $\RR_B$ implies that induced filtrations on submodules and quotient filtrations are good. For a short exact sequence equipped with these filtrations, both the Rees sequence and the associated graded sequence are exact. These are the properties of filtered modules used in \cite[Section~4.1]{SW08}; the general theory of strict filtered sequences and good filtrations is developed in \cite{Laum83,Bjo93}.

We next recall the characteristic variety and the cycle associated with a good filtration.

\begin{definition}\label{def:characteristic-cycle}
Let $N$ be a finite $D$-module with a good $L$-filtration $F$. Its \emph{$L$-characteristic variety} is $\Ch^L(N)=\Supp\gr_FN$. When this support is contained in a finite union of distinct irreducible subvarieties $\Lambda_i$ of dimension $n$, its \emph{$L$-characteristic cycle} is
\begin{equation}\label{eq:cc-length}
\CC^L(N)=\sum_i\ell_{W_{P_i}}(\gr_FN)_{P_i}[\Lambda_i],
\quad P_i=I(\Lambda_i).
\end{equation}
Here $\ell$ denotes module length, and a component that does not occur is assigned coefficient zero.
\end{definition}

Every nonempty irreducible component of $\Ch^L(N)$ has dimension at least $n$ \cite[Theorem~3.2]{SW08}. This dimension bound reflects the involutivity of characteristic varieties \cite{Gab81}. The cycle is independent of the good filtration and is additive in short exact sequences \cite[Definition~4.4 and Lemma~4.5]{SW08}. The same length convention applies over a field extension of $\C$. In particular, these statements apply to the fibers and the subquotients used below. Filtration shifts do not change these ungraded lengths.

The dependence of initial modules on the weight is organized by Gr\"obner fans; for Weyl-algebra modules, see \cite{ACG00}, and for the hypergeometric constructions see \cite{SST00}. We keep $L$ fixed when changing the parameter ring. We next record precisely what happens under that change.

\begin{lemma}\label{lem:basechange}
Let $B\to B'$ be a homomorphism of Noetherian commutative $\C$-algebras, and let $M$ be a finite $D_B$-module with a good filtration $F$. Equip $M'=M\otimes_BB'$ with the filtration
\[
F_kM'=\im(F_kM\otimes_BB'\longrightarrow M').
\]
Then this filtration is good, and there is a natural surjection
\begin{equation}\label{eq:basechange-surj}
(\gr_FM)\otimes_BB'\longrightarrow\gr_FM'.
\end{equation}
More precisely, if $N=\mathscr M\otimes_BB'$, then
\begin{equation}\label{eq:basechange-rees}
\Rees_FM'=N/H^0_{(h)}N.
\end{equation}
If $B'$ is flat over $B$, the map in \eqref{eq:basechange-surj} is an isomorphism.
\end{lemma}

\begin{proof}
After inverting $h$, the tensor product $N$ becomes $M'[h,h^{-1}]$. The image of $N$ in this localization has degree-$k$ part equal to the image of $F_kM\otimes_BB'$ in $M'$, multiplied by $h^k$. It is therefore the Rees module of the stated filtration. The kernel of the localization map is $H^0_{(h)}N$, which proves \eqref{eq:basechange-rees}. Since $N$ is finite over $\RR_{B'}$, its image is finite, and the filtration is good. Passing to the quotient by $h$ gives \eqref{eq:basechange-surj} by right exactness. If $B'$ is flat, the injectivity of multiplication by $h$ on $\mathscr M$ is preserved, and $H^0_{(h)}N=0$.
\end{proof}

The surjection in \eqref{eq:basechange-surj} preserves the following support condition. Fix distinct irreducible subvarieties $\Lambda_1,\ldots,\Lambda_q$ of dimension $n$ in $\Spec W$, put $\Lambda=\bigcup_i\Lambda_i$, and write $P_i=I(\Lambda_i)$.

\begin{definition}\label{def:relative-support}
A finite $D_B$-module $M$ with a good filtration $F$ satisfies the \emph{product support condition} for $\Lambda$ if
\begin{equation}\label{eq:product-support}
\Supp\gr_FM\subseteq\Lambda\times\Spec B.
\end{equation}
For such a module and a closed point $b\in\Spec B$, its fiber multiplicities are
\begin{equation}\label{eq:mu-family}
\mu_i(M_b)=\ell_{W_{P_i}}(\gr_FM_b)_{P_i}.
\end{equation}
Here $M_b=M\otimes_B\kappa(b)$ has the induced filtration of Lemma~\ref{lem:basechange}. If $B$ is a domain with fraction field $K$, we write $M_K=M\otimes_B K$ and define its generic multiplicities $\mu_i(M_K)$ using the corresponding primes $P_iW_K$. The specialization homomorphism $\spc_\Lambda$ is defined by $\spc_\Lambda[\Lambda_i\times\Spec K]=[\Lambda_i]$.
\end{definition}

Throughout Sections~\ref{sec:dvr} and~\ref{sec:constructibility} we assume \eqref{eq:product-support}. Equivalently, every prime in the relative associated graded support contains the extension of at least one of the ideals $P_i$. The condition is preserved by base change with the filtration of Lemma~\ref{lem:basechange}. It is also preserved by submodules and quotients with their induced good filtrations. The ideals $P_iW_K$ are prime because an irreducible variety over $\C$ is geometrically integral. Thus the generic and special characteristic cycles are compared in the free group generated by the symbols $[\Lambda_i]$.

\begin{remark}\label{rem:filtration-choice}
The support condition \eqref{eq:product-support} is imposed on a good filtration of the relative module $M$. By Lemma~\ref{lem:basechange}, the induced filtrations on the fibers are good and satisfy the same support condition. The multiplicities in \eqref{eq:mu-family} are independent of the choice of good filtration on each fiber.
\end{remark}

\begin{example}\label{ex:bad-support}
Let $R=\C[t]_{(t)}$, take one differential variable $x$, and use the order filtration. Consider
\[
M=D_R/D_R(t\partial_x-1).
\]
Over $K$, this is a rank-one connection, whereas $M/tM=0$. This family occurs in \cite[Remark~2.7]{MMW05}. Its relative associated graded module has support $V(t\xi)$: since the associated graded ring is a domain, the initial ideal of the principal left ideal is generated by $t\xi$. Over $t=0$ this support contains the entire cotangent plane. It is therefore not contained in a product with a fixed finite union of curves. This explains why fiberwise holonomicity alone does not give the hypothesis of Theorem~\ref{thm:general}.
\end{example}


\section{Specialization over a discrete valuation ring}\label{sec:dvr}

Let $R$ be the local ring of a smooth irreducible complex algebraic curve at a closed point. Write $t$ for a uniformizer, $k=R/(t)=\C$, and $K=\Frac R$. In this section the Rees parameter is always $h$. We prove the following form of \eqref{eq:intro-tor}.

\begin{theorem}\label{thm:curve}
Let $M$ be a finite $D_R$-module with a good $L$-filtration satisfying the product support condition \eqref{eq:product-support}. If $T=H^0_{(t)}M$, then
\begin{equation}\label{eq:curve-tor}
\CC^L(M/tM)-\spc_\Lambda\CC^L(M_K)=\CC^L(T/tT)=\CC^L\bigl(\Tor_1^R(M,k)\bigr).
\end{equation}
Consequently, $\mu_i(M/tM)\ge\mu_i(M_K)$ for every $i$.
\end{theorem}

The proof first treats the quotient $M/T$, which is torsion-free over $R$. 
Comparisons of the classes of filtered kernels and cokernels appear in \cite[Lemma~3.5.13]{Laum83} and \cite[Appendix~IV, Proposition~3.13]{Bjo93}, and are also used in \cite[Proposition~2.7]{Wu22}. We write the two-parameter calculation explicitly because $t$ and $h$ need not form a regular sequence on the Rees module.


\subsection{Two central parameters}

The next lemma gives an identity in the Grothendieck group of finitely generated modules. We will apply it to the Rees parameter $h$ and the curve parameter $t$.

\begin{lemma}\label{lem:nilpotent}
Let $\mathscr A$ be a left Noetherian ring, let $z\in\mathscr A$ be central, and let $V$ be a finite $\mathscr A$-module such that $z^aV=0$ for some $a\ge1$. In the Grothendieck group of finite $\mathscr A/(z)$-modules,
\begin{equation}\label{eq:nilpotent-class}
[V/zV]=[0:_Vz].
\end{equation}
The same identity holds for any additive invariant defined on the modules appearing in the proof.
\end{lemma}

\begin{proof}
There is an exact sequence of modules annihilated by $z$,
\begin{equation}\label{eq:nilpotent-exact}
0\longrightarrow0:_{zV}z\longrightarrow0:_Vz\longrightarrow V/zV
\xrightarrow{\ z\ }zV/z^2V\longrightarrow0.
\end{equation}
To verify exactness, first observe that the kernel of the natural map $0:_Vz\to V/zV$ is $(0:_Vz)\cap zV=0:_{zV}z$. If $v+zV$ belongs to the kernel of multiplication by $z$ on $V/zV$, then $zv=z^2w$ for some $w\in V$. Hence $v-zw\in0:_Vz$, and its image in $V/zV$ is $v+zV$. Finally, the map $V/zV\to zV/z^2V$ is surjective by definition. This proves the exactness of \eqref{eq:nilpotent-exact}.

Additivity in \eqref{eq:nilpotent-exact} gives
\[
[0:_Vz]-[V/zV]=[0:_{zV}z]-[zV/z^2V].
\]
Repeat this equality for $zV,z^2V,\ldots,z^{a-1}V$. The last difference is zero, proving \eqref{eq:nilpotent-class}.
\end{proof}

If $z$ is homogeneous of nonzero degree, the multiplication map in \eqref{eq:nilpotent-exact} becomes homogeneous of degree zero after shifting its target. We apply Lemma~\ref{lem:nilpotent} after forgetting the grading, since grading shifts do not affect the lengths defining the characteristic cycles in \eqref{eq:cc-length}.

\begin{lemma}\label{lem:two-parameters}
Let $t,h$ be central elements of a left Noetherian ring $\mathscr A$, and let $\mathscr Q$ be a finite $\mathscr A$-module on which multiplication by $t$ and multiplication by $h$ are injective. Set
\[
G=\mathscr Q/h\mathscr Q,
\quad N=\mathscr Q/t\mathscr Q,
\quad U=H^0_{(h)}N,
\quad V=N/U.
\]
Then $U$ is annihilated by a power of $h$, $V$ is $h$-torsion-free, and
\begin{equation}\label{eq:two-exact}
0\longrightarrow U/hU\longrightarrow G/tG\longrightarrow V/hV\longrightarrow0
\end{equation}
is exact. There is a natural isomorphism, up to a grading shift,
\begin{equation}\label{eq:two-kernels}
0:_Nh\simeq0:_Gt.
\end{equation}
In the Grothendieck group of finite $\mathscr A/(t,h)$-modules,
\begin{equation}\label{eq:two-class}
[V/hV]=[G/tG]-[0:_Gt].
\end{equation}
\end{lemma}

\begin{proof}

Since $N$ is Noetherian, the ascending chain
\[
0:_Nh\subseteq0:_Nh^2\subseteq\cdots
\]
stabilizes. Thus $h^aU=0$ for some $a\ge1$. If $hn\in U$, then $h^{a+1}n=0$, so $n\in U$. It follows that $U\cap hN=hU$ and that $V=N/U$ has no $h$-torsion. Consequently, reducing the exact sequence $0\to U\to N\to V\to0$ modulo $h$ gives
\[
0\longrightarrow U/hU\longrightarrow N/hN
\longrightarrow V/hV\longrightarrow0.
\]
Since $N/hN\simeq G/tG$, this proves \eqref{eq:two-exact}.

We next prove \eqref{eq:two-kernels}. Let $\bar a\in0:_Nh$ be represented by $a\in\mathscr Q$. Since $h\bar a=0$, there exists $b\in\mathscr Q$ such that $ha=tb$. Moreover, $b$ is unique because multiplication by $t$ is injective on $\mathscr Q$. Define a map $0:_Nh\to0:_Gt$ by sending $\bar a$ to the class $\bar b$ of $b$ in $G$. If $a$ is replaced by $a+td$, then $b$ is replaced by $b+hd$, so the map is well defined. If $\bar b=0$, then $b=hd$ for some $d\in\mathscr Q$. Hence $h(a-td)=0$, and the injectivity of $h$ gives $a=td$. This proves injectivity. Conversely, given $\bar b\in0:_Gt$, we have $tb=ha$ for some $a\in\mathscr Q$. The class of $a$ in $N$ belongs to $0:_Nh$ and maps to $\bar b$, proving surjectivity. Since $t$ and $h$ are central, the map is linear over $\mathscr A/(t,h)$.

Since $V=N/U$ has no $h$-torsion, we have $0:_Nh=0:_Uh$. As $h^aU=0$ for some $a\ge1$, Lemma~\ref{lem:nilpotent} gives $[U/hU]=[0:_Uh]$. Taking equivalence classes in \eqref{eq:two-exact} and using \eqref{eq:two-kernels}, we obtain \eqref{eq:two-class}. 
\end{proof}


\subsection{The torsion-free quotient}

Let $Q$ be a finite $D_R$-module with no $R$-torsion, and choose a good filtration satisfying the product
support condition \eqref{eq:product-support}. Since
\[
\mathscr Q=\Rees_FQ\subseteq Q[h,h^{-1}],
\]
both $t$ and $h$ act injectively on $\mathscr Q$. Apply Lemma~\ref{lem:two-parameters} to $\mathscr A=\RR_R$. By Lemma~\ref{lem:basechange},
\[
V=N/H^0_{(h)}N=\Rees_F(Q/tQ)
\]
for the quotient filtration. We obtain
\begin{equation}\label{eq:defect-sequence}
0\longrightarrow U/hU\longrightarrow G/tG
\longrightarrow\gr_F(Q/tQ)\longrightarrow0,
\end{equation}
and
\begin{equation}\label{eq:defect-classes}
[\gr_F(Q/tQ)]=[G/tG]-[0:_Gt].
\end{equation}
All three terms in \eqref{eq:defect-classes} are finite $W$-modules supported on $\Lambda$. For the proof of the class identity, the modules obtained by replacing $U$ with $h^jU$ also have this support: each $h^jU/h^{j+1}U$ is a quotient, up to a shift, of $U/hU$, and $0:_{h^jU}h$ is a submodule of $0:_Uh$. Thus localization of the identity at any $P_i$ involves only finite lengths.

Put
\[
\mathfrak p_i=P_iW_R,
\quad \mathfrak q_i=P_iW_R+(t).
\]
Since $(W_R)_{\mathfrak q_i}/(t)=W_{P_i}$, equation~\eqref{eq:defect-classes} gives
\begin{equation}\label{eq:defect-length}
\mu_i(Q/tQ)=
\ell_{(W_R)_{\mathfrak q_i}}(G/tG)_{\mathfrak q_i}
-\ell_{(W_R)_{\mathfrak q_i}}(0:_Gt)_{\mathfrak q_i}.
\end{equation}
We next compare the difference in \eqref{eq:defect-length} with the generic multiplicity. We use the Koszul Euler characteristic of multiplication by $t$ from multiplicity theory \cite[Chapter~V]{Ser00}.

\begin{definition}[Local Euler characteristic]\label{def:local-euler-characteristic}
Let $V$ be a finite module over a Noetherian local ring, and suppose that the kernel and cokernel of multiplication by $t$ on $V$ have finite length. The \emph{Euler characteristic of multiplication by $t$} is
\[
\chi_t(V)=\ell(V/tV)-\ell(0:_Vt).
\]
\end{definition}

\begin{lemma}\label{lem:local-euler}
Let $G$ be a finite $W_R$-module satisfying
\[
\Supp G\subseteq\bigcup_i V(P_iW_R).
\]
Then the kernel and cokernel of multiplication by $t$ on $G_{\mathfrak q_i}$ have finite length, and
\begin{equation}\label{eq:local-euler}
\ell(G_{\mathfrak q_i}/tG_{\mathfrak q_i})
-\ell(0:_{G_{\mathfrak q_i}}t)
=\ell_{(W_R)_{\mathfrak p_i}}G_{\mathfrak p_i}. 
\end{equation}
\end{lemma}

\begin{proof}
If $P_jW_R\subseteq\mathfrak q_i$, then contraction to $W$ gives $P_j\subseteq P_i$. The ideals $P_i$ define distinct irreducible subvarieties of the same dimension, so $j=i$. Thus only the component indexed by $i$ survives in the support after localization at $\mathfrak q_i$.

The ring
\begin{equation}\label{eq:component-dvr}
\left((W/P_i)\otimes_\C R\right)_{(t)}
\end{equation}
is a domain because $W/P_i$ is geometrically integral over $\C$ and $(W/P_i)\otimes_\C R$ embeds into the domain $(W/P_i)\otimes_\C K$. It is a Noetherian local ring whose maximal ideal is generated by the nonzero element $t$, and hence a discrete valuation ring \cite[Tag~00PD]{Stacks}. Its residue field is $\Frac(W/P_i)$, and its only prime ideals are $(0)$ and $(t)$. By the support condition and the preceding argument, the only possible primes in $\Supp G_{\mathfrak q_i}$ are $\mathfrak p_i(W_R)_{\mathfrak q_i}$ and $\mathfrak q_i(W_R)_{\mathfrak q_i}$.

Write $A_i=(W_R)_{\mathfrak q_i}$. By the existence of finite prime filtrations \cite[Tag~00L0]{Stacks}, $G_{\mathfrak q_i}$ has a finite filtration whose successive quotients are isomorphic to
\[
A_i/\mathfrak p_iA_i
\quad\text{or}\quad A_i/\mathfrak q_iA_i.
\]
On a factor of the first type, multiplication by $t$ is injective, with cokernel of length one. On a factor of the second type, its kernel and cokernel both have length one. Thus all the required lengths are finite.

By the kernel--cokernel exact sequence for multiplication by $t$, the Euler characteristic $\chi_t$ is additive on short exact sequences of modules with this support. Hence $\chi_t(G_{\mathfrak q_i})$ equals the number of factors of the first type. Localizing the prime filtration at $\mathfrak p_i$ annihilates every factor of the second type, while each factor of the first type becomes the residue field of $(W_R)_{\mathfrak p_i}$. It follows that $\chi_t(G_{\mathfrak q_i})=\ell_{(W_R)_{\mathfrak p_i}}G_{\mathfrak p_i}$, as claimed.
\end{proof}

\begin{proposition}\label{prop:torsionfree-constant}
If $Q$ has no $R$-torsion and satisfies \eqref{eq:product-support}, then
\begin{equation}\label{eq:torsionfree-cycle}
\CC^L(Q/tQ)=\spc_\Lambda\CC^L(Q_K).
\end{equation}
\end{proposition}

\begin{proof}
For each $i$, combining \eqref{eq:defect-length} with Lemma~\ref{lem:local-euler}, we obtain $\mu_i(Q/tQ)=\ell_{(W_R)_{\mathfrak p_i}}G_{\mathfrak p_i}$. Since $R\to K$ is a localization, Lemma~\ref{lem:basechange} gives
\[
G\otimes_RK=\gr_F(Q_K).
\]
Moreover, $(W_R)_{\mathfrak p_i}=(W_K)_{P_iW_K}$, so $\ell_{(W_R)_{\mathfrak p_i}}G_{\mathfrak p_i}=\mu_i(Q_K)$. This proves \eqref{eq:torsionfree-cycle}.
\end{proof}

Although $G_{\mathfrak q_i}$ may have dimension one and need not have finite length, the kernel and cokernel of multiplication by $t$ on this module have finite length. In the prime filtration above, the factors $A_i/\mathfrak q_iA_i$ contribute equally to these two lengths and therefore cancel in the Euler characteristic $\chi_t(G_{\mathfrak q_i})$. The remaining contribution is precisely $\ell_{(W_R)_{\mathfrak p_i}}G_{\mathfrak p_i}$.

\begin{example}\label{ex:rees-defect}
Let $R=\C[t]_{(t)}$ and let $Q=R[x]$ have its usual $D_R$-module structure. Choose positive integers $a,m$. The submodule
\[
\mathscr Q=(t^a,h^m)R[x,h]\subseteq Q[h,h^{-1}]
\]
is the Rees module of a good order filtration on $Q$. It is stable under $x$ and $h\partial_x$, and becomes all of $Q[h,h^{-1}]$ after $h$ is inverted. Let $e_0=t^a$ and $e_1=h^m$. These generators satisfy $h^me_0=t^ae_1$ over $R[x,h]$. Hence, with grading shifts suppressed,
\[
G=\mathscr Q/h\mathscr Q\simeq R[x]e_0\oplus(R/(t^a))[x]e_1,
\]
and
\[
N=\mathscr Q/t\mathscr Q\simeq\C[x,h]/(h^m)e_0\oplus\C[x,h]e_1.
\]
The order symbol $\xi$ kills $G$, so the product support condition holds with the zero section. The $h$-torsion of $N$ is its first summand. At the generic point of the zero section of the special fiber,
\[
\ell(G/tG)=2,
\quad\ell(0:_Gt)=1,
\quad\ell(U/hU)=1.
\]
Since $Q_K$ and $Q/tQ$ both have multiplicity one, \eqref{eq:defect-length} reads $2-1=1$. The natural surjection from $G/tG$ to $\gr_F(Q/tQ)$ has a nonzero kernel even though $Q$ is flat over $R$.
\end{example}


\subsection{Parameter torsion and the characteristic cycle}

We now prove Theorem~\ref{thm:curve}. The $t$-torsion contributes the effective term in the specialization formula.

\begin{proof}[Proof of Theorem~\ref{thm:curve}]
By Noetherianity, $T=H^0_{(t)}M$ is a finite submodule of $M$ annihilated by some power $t^a$. Equip $T$ with the induced good filtration and $Q=M/T$ with the quotient filtration. Their associated graded modules satisfy the product
support condition \eqref{eq:product-support}.

Since $Q$ is torsion-free over the discrete valuation ring $R$, it is $R$-flat \cite[Tag~0539]{Stacks}. Tensoring $0\to T\to M\to Q\to0$ with $k$ therefore gives
\begin{equation}\label{eq:special-torsion-exact}
0\longrightarrow T/tT\longrightarrow M/tM\longrightarrow Q/tQ\longrightarrow0.
\end{equation}
Since $M_K=Q_K$, additivity of characteristic cycles and Proposition~\ref{prop:torsionfree-constant} give
\[
\CC^L(M/tM)-\spc_\Lambda\CC^L(M_K)=\CC^L(T/tT).
\]

The free resolution $0\to R\xrightarrow{t}R\to k\to0$ gives
\[
\Tor_1^R(M,k)=0:_Mt=0:_Tt.
\]
Apply Lemma~\ref{lem:nilpotent} with $z=t$ to $T$. The modules in its exact sequences are finite $D_R/(t)$-modules. With induced or quotient filtrations, their associated graded supports are contained in $\Lambda$, so characteristic cycles define an additive invariant on them. We conclude that
\[
\CC^L(T/tT)=\CC^L(0:_Tt)=\CC^L\bigl(\Tor_1^R(M,k)\bigr).
\]
Every coefficient of this last cycle is a nonnegative length. This proves \eqref{eq:curve-tor} and the asserted inequalities.
\end{proof}

\begin{corollary}\label{cor:flat-curve}
If $M$ is $R$-flat and satisfies the product support condition \eqref{eq:product-support}, its generic and special $L$-characteristic cycles have the same coefficients on the fixed components.
\end{corollary}

\begin{proof}
Flatness makes $\Tor_1^R(M,k)$ vanish, so the assertion follows from \eqref{eq:curve-tor}.
\end{proof}

In Theorem~\ref{thm:curve}, parameter torsion is computed after pullback along a parameter curve. We now give a flatness criterion in terms of the characteristic cycle for the order filtration.

\begin{corollary}\label{cor:flatness}
Assume \eqref{eq:product-support} for the order filtration. Then
\[
M_R\text{ is }R\text{-flat}
\quad\Longleftrightarrow\quad
\CC^F(M_R/tM_R)=\spc_\Lambda\CC^F(M_R\otimes_RK).
\]
\end{corollary}

\begin{proof}
The module $\Tor_1^R(M_R,k)$ is finite over $D$, and its induced good order filtration has characteristic support contained in $\Lambda$. Hence it is holonomic. A nonzero holonomic module has a nonempty characteristic variety of dimension $n$, with positive multiplicities \cite[Chapter~3]{HTT08}. By Theorem~\ref{thm:curve}, equality of the cycles is therefore equivalent to $\Tor_1^R(M_R,k)=0$. Since $\Tor_1^R(M_R,k)=0:_{M_R}t$, its vanishing is equivalent to torsion-freeness and hence to flatness over $R$ \cite[Tag~0539]{Stacks}.
\end{proof}

\begin{example}\label{ex:rank-flatness}
Let $R=\C[t]_{(t)}$ and $\delta_0=D_{\A^1}/D_{\A^1}x$. The family
\[
M_R=(R\otimes_\C\OO_{\A^1})\oplus((R/(t))\otimes_\C\delta_0)
\]
has rank one on both fibers, whereas their characteristic cycles for the order filtration are $[T^*_{\A^1}\A^1]$ and $[T^*_{\A^1}\A^1]+[T_0^*\A^1]$, respectively. Its first Tor module is $\delta_0$. Thus the full characteristic cycle detects a failure of flatness that the rank does not detect.
\end{example}

\begin{example}\label{ex:curve-torsion}
Let $B=\C[b_1,b_2]$, $J=(b_1,b_2)$, and $M=B[x]\otimes_BJ$, with the usual differential action on $B[x]$. The module is torsion-free over $B$ and has relative characteristic support in the zero section for the order filtration. Its fibers are one copy of $\OO_{\A^n}$ off the origin and two copies at the origin. For the curve $b_1=t^a$, $b_2=t^b$, with $a,b>0$, the presentation with relation $b_2e_1-b_1e_2=0$ gives
\[
J\otimes_BR\simeq R\oplus R/(t^{\min(a,b)}),
\quad R=\C[t]_{(t)}.
\]
Thus $\Tor_1^R(M_R,\C)\simeq\C[x]$, and the characteristic cycle gains one copy of the zero section. Torsion-freeness over the parameter plane does not imply torsion-freeness after restriction to a curve through its nonflat point.
\end{example}


\section{Constructibility and upper semicontinuity}\label{sec:constructibility}

Throughout this section, $B$ is a finitely generated commutative $\C$-algebra, and $M$ is a finite $D_B$-module with a good filtration $F$ satisfying the product support condition \eqref{eq:product-support}. We first prove constructibility of the fiber multiplicities in \eqref{eq:mu-family}. Together with the curve formula, this shows that their upper level sets are Zariski closed. For hypergeometric systems, stratifications of the parameter space also arise in the study of exceptional parameters \cite{Mat03} and the variation of series solutions \cite{BFM16,BFM18}. The argument below uses the relative associated graded module and applies to any family satisfying \eqref{eq:product-support}.

\begin{proposition}\label{prop:dense-open}
    Assume that $B$ is a domain, with fraction field $K$. Then there is a nonzero $f\in B$ such that, for every closed point $b\in\Spec B_f$ and every $i$,
    \begin{equation}\label{eq:dense-open-constant}
    \mu_i(M_b)=\mu_i(M_K),
    \end{equation}
    where $M_b=M\otimes_B\kappa(b)$ and $M_K=M\otimes_BK$. 
\end{proposition}
    
\begin{proof}
    Write $\mathscr M=\Rees_FM$ and $G=\gr_FM$. The $B$-algebra $W_B$ is finitely presented and Noetherian, and $G$ is a finitely presented $W_B$-module. It follows from generic freeness \cite[Tag~051S]{Stacks} that there is a nonzero $f_0\in B$ for which $G_{f_0}$ is free, hence flat, over $B_{f_0}$.

    For $b\in\Spec B_{f_0}$, tensor the exact sequence
    \[
    0\longrightarrow\mathscr M\xrightarrow{\ h\ }\mathscr M\longrightarrow G\longrightarrow0
    \]
    with $\kappa(b)$. The vanishing of $\Tor_1^{B_{f_0}}(G_{f_0},\kappa(b))$ implies that $h$ is injective on $\mathscr M\otimes_B\kappa(b)$. Lemma~\ref{lem:basechange} now gives
    \begin{equation}\label{eq:generic-gr-fiber}
    G\otimes_B\kappa(b)\simeq\gr_F(M_b).
    \end{equation}
    Similarly, $G\otimes_B K\simeq\gr_F(M_K)$. We next compare the multiplicities of the generic fiber with those of the fibers over $\Spec B_{f_0}$.
    
    Choose a finite prime filtration of $G$ whose successive quotients are of the form $W_B/\mathfrak r_j$, where each $\mathfrak r_j$ is a prime ideal of $W_B$ \cite[Tag~00L0]{Stacks}. Every $\mathfrak r_j$ contains some $P_iW_B$, by \eqref{eq:product-support}. Apply generic freeness to the finitely many factors, and shrink the open set so that all surviving factors are flat over the parameter ring. The filtration then remains exact after passage to every fiber over that open set. We now compute the contribution of each factor to the multiplicities.

    If $\mathfrak r_j\cap B\ne0$, invert a nonzero element of this intersection, so that the factor vanishes. If $\mathfrak r_j=P_iW_B$, the factor contributes one at $P_i$ in every fiber and zero at the other component primes. It remains to consider
    \[
    P_iW_B\subsetneq\mathfrak r_j.
    \]
    Choose $g\in\mathfrak r_j\setminus P_iW_B$. Its nonzero image in $(W/P_i)\otimes_\C B$ can be written
    \begin{equation}\label{eq:coefficients}
    \overline g=\sum_{a=1}^m c_a\otimes b_a,
    \end{equation}
    where the $c_a\in W/P_i$ are linearly independent over $\C$ and the $b_a$ belong to $B$. Choose a nonzero $b_a$ and invert it. By linear independence, $\overline g$ remains nonzero in every closed fiber over the resulting open set. Thus the factor has zero localization at $P_i$ in every such fiber. For $k\ne i$, the primes $P_i$ and $P_k$ are distinct and have the same dimension, so we may choose $p\in P_i\setminus P_k$. This element annihilates the factor and becomes a unit upon localization at $P_k$, so the localization there also vanishes. The same conclusions hold for the generic fiber.

    Since the prime filtration has finitely many factors, we may make all these localizations of $B$ simultaneously. On the resulting open set, the length at $P_i$ in each closed fiber equals the number of factors isomorphic to $W_B/P_iW_B$ in the prime filtration. The same number gives the length at $P_iW_K$ in the generic fiber. By \eqref{eq:generic-gr-fiber} and the corresponding isomorphism over $K$, we obtain \eqref{eq:dense-open-constant}.
\end{proof}

\begin{corollary}\label{cor:finite-strata}
    There is a finite partition of $\Spec B$ into locally closed algebraic strata on which all the functions $\mu_i(M_b)$ are constant. On every irreducible subvariety of the parameter space, the generic multiplicity vector equals the vector on a nonempty open subset of that subvariety.
\end{corollary}

\begin{proof}
    Passing from $B$ to its reduction does not change the closed fibers. Apply Proposition~\ref{prop:dense-open} to the pullback to each irreducible component of $\Spec B_{\mathrm{red}}$. By removing the intersections with the other components, we may choose the resulting open subsets to be disjoint. Their complement is a closed set of strictly smaller dimension. Repeating the construction on this complement, we obtain a finite stratification, since there are only finitely many irreducible components at each stage and the dimension decreases. The product support condition \eqref{eq:product-support} is preserved under each pullback by Lemma~\ref{lem:basechange}. The final assertion follows by applying Proposition~\ref{prop:dense-open} to the chosen irreducible subvariety.
\end{proof}

The following consequence of Noether normalization allows us to apply the curve formula at points on the boundaries of strata.

\begin{lemma}\label{lem:curve-selection}
    Let $Y$ be an irreducible complex algebraic variety of positive dimension, $U$ a nonempty open subset of $Y$, and $y$ a closed point of $Y$. Then there exists an irreducible algebraic curve in $Y$ passing through $y$ and meeting $U$.
\end{lemma}

\begin{proof}
    We may replace $Y$ by an affine open neighborhood of $y$ and $U$ by its intersection with this neighborhood. By Noether normalization \cite[Tag~00OW]{Stacks}, there is a finite morphism $\pi:Y\to\A^m$, where $m=\dim Y$. The image of $Y\setminus U$ is closed and has dimension less than $m$. Choose a line $\ell$ through $\pi(y)$ that is not contained in this image. Since $\pi^{-1}(\ell)$ is cut out by $m-1$ equations, the height theorem gives an irreducible component through $y$ of dimension at least one. Finiteness of $\pi$ implies that this component has dimension at most one. With its reduced structure, it is a curve mapping onto $\ell$ and hence meets $U$.
\end{proof}

\begin{proof}[Proof of Theorem~\ref{thm:general}]
    By Corollary~\ref{cor:finite-strata}, the parameter space admits a finite partition into irreducible locally closed strata $S_\alpha$ on which the multiplicity vector is constant. Fix $i$ and an integer $r$, and consider a stratum $S_\alpha$ on which $\mu_i\ge r$. We show that the same inequality holds at every closed point $b_0\in\overline{S_\alpha}$.
    
    If $\overline{S_\alpha}$ has dimension zero, there is nothing to prove. Otherwise, $S_\alpha$ is a nonempty open subset of its irreducible closure. By Lemma~\ref{lem:curve-selection}, there is an irreducible curve through $b_0$ meeting $S_\alpha$. Normalize this curve and choose a closed point lying over $b_0$. The local ring $R$ at this point is a discrete valuation ring with residue field $\C$, and its generic point maps to $S_\alpha$.
    
    Equip $M_R=M\otimes_BR$ with the induced filtration. By Lemma~\ref{lem:basechange}, the product support condition holds for $M_R$. By Proposition~\ref{prop:dense-open}, the generic multiplicities on the curve agree with the constant multiplicities on $S_\alpha$, since the equality holds on a nonempty open subset of the curve meeting $S_\alpha$. Applying Theorem~\ref{thm:curve}, we obtain
    \[
    \mu_i(M_{b_0})\ge\mu_i((M_R)_K)=\mu_i(M_b)\ge r
    \quad(b\in S_\alpha).
    \]
    It follows that, on closed points, the upper level set is
    \[
    \{b:\mu_i(M_b)\ge r\}
    =\bigcup_{\alpha:\,\mu_i|_{S_\alpha}\ge r}\overline{S_\alpha}.
    \]
    This is a finite union of Zariski closed sets, proving upper semicontinuity.
\end{proof}

\begin{corollary}\label{cor:generic-min}
    If $B$ is a domain and $K=\Frac B$, then $\mu_i(M_b)\ge\mu_i(M_K)$ at every closed point $b$. If the family is flat over a smooth curve, the multiplicity vector is constant on that curve.
\end{corollary}

\begin{proof}
    The first assertion follows from upper semicontinuity and the dense open subset given by Proposition~\ref{prop:dense-open}. The second follows by applying Corollary~\ref{cor:flat-curve} at every closed point of the curve.
\end{proof}

\begin{proposition}\label{prop:positive-functional}
    Let $L_1,\ldots,L_m$ be fixed projective weights such that the same finite $D_B$-module $M$ satisfies \eqref{eq:product-support} for each weight, with a possibly different fixed set of components. Let $\Phi$ be a fixed integer linear combination of the corresponding multiplicities. Suppose that $\Phi(\Tor_1^R(M_R,k))\ge0$ for every DVR arising from pullback along a parameter curve. Then $b\mapsto\Phi(M_b)$ is Zariski upper semicontinuous and constant on a finite algebraic stratification.
\end{proposition}

\begin{proof}   
    Take a common refinement of the finitely many stratifications given by Theorem~\ref{thm:general}. On each stratum, the expression $\Phi(M_b)$ is constant. Along a parameter curve, we apply Theorem~\ref{thm:curve} at each weight and take the indicated linear combination to obtain
    \[
    \Phi(M_0)-\Phi(M_{\mathrm{gen}})=\Phi\bigl(\Tor_1^R(M_R,k)\bigr)\ge0.
    \]
    The same unfiltered Tor module occurs at every weight. By Proposition~\ref{prop:dense-open}, the generic cycle expression agrees with its value at general complex fibers. The curve-selection argument in the proof of Theorem~\ref{thm:general} then shows that every superlevel set of $\Phi$ is closed.
\end{proof}


\section{Hypergeometric systems}\label{sec:gkz}

We return to the matrix $A$ from Theorem~\ref{thm:gkz}. We first recall the conormal varieties indexing the characteristic components and then verify the relative support condition using the symbol description of \cite{SW08}. For affine semigroup rings, we follow \cite[Chapter~7]{MS05}; for toric initial ideals and subdivisions, we refer to \cite[Chapters~4 and~8]{Stu96} and \cite{SST00}.

A face of a convex cone or polytope is its intersection with a supporting hyperplane, and a facet is a face of codimension one. A \emph{face of $A$} is the set of all columns of $A$ lying in a face of the cone $\R_{\ge0}A$. For such a face $G$, write $S_G=\C[\N G]$ for its semigroup ring. The torus $T_A=\operatorname{Hom}(\Z A,\C^*)$ acts on $\A^n_\xi$ with weights $a_j$. The conormal varieties of these orbits index the components of the characteristic cycles.

\begin{definition}[Conormals to torus orbits]\label{def:orbit-conormal}
For a subset $\tau$ of the column indices, let $\mathbf1_\tau$ denote the point whose coordinates are one on $\tau$ and zero elsewhere, and put $O_A^\tau=T_A\cdot\mathbf1_\tau$. In the cotangent coordinates $(x,\xi)$ of \cite{SW08}, define
\begin{equation}\label{eq:conormal-orbit}
C_A^\tau=\{(x,\xi):\xi\in O_A^\tau,\ \textstyle\sum_j a_{ij}x_j\xi_j=0\text{ for }1\le i\le d\},
\quad \Lambda_\tau=\overline{C_A^\tau}.
\end{equation}
\end{definition}

For $\tau=\varnothing$, we obtain the zero section $\{\xi=0\}$. We next recall the umbrella, which determines the components associated with a given weight.

\begin{definition}[$(A,L)$-umbrella]\label{def:umbrella}
Let $L=(u,v)$ be a rational projective weight. Consider the points $[1:0]$ and $[v_j:a_j]$ in $\PP^d(\R)$. Since the cone of $A$ is pointed and its columns are nonzero, there is a rational linear form $\ell$ that is positive on every $a_j$. For sufficiently small positive rational $\varepsilon$, the functional
\[
(t,a)\longmapsto\varepsilon t+\ell(a)
\]
is positive on the displayed representatives. Normalize these representatives so that $\varepsilon t+\ell(a)=1$, and take their convex hull. The \emph{$(A,L)$-umbrella} $\Phi_A^L$ consists of the faces of this convex hull that do not contain $[1:0]$, together with the empty face. Each face is labeled by all the columns lying on it.
\end{definition}

This is the projective convex-hull construction of \cite[Definition~2.7]{SW08}. If every $v_j>0$, this construction gives the familiar collection of faces of
\[
\conv(0,a_1/v_1,\ldots,a_n/v_n)
\]
that do not contain the origin. The projective construction allows zero and negative $v_j$ without changing the given weight. The convention of labeling each face by all the columns lying on it agrees with the subdivision convention for initial toric ideals in \cite[Chapter~8]{Stu96}; wall weights are not replaced by nearby generic weights.

We write $\operatorname{in}_v(I_A)$ for the initial ideal of $I_A$ with respect to the weights $v_j$ on the variables $\partial_j$, and let $\sigma_L$ denote the principal symbol with respect to $L$.
Put $e_i=\sum_j a_{ij}x_j\xi_j\in W$ and
\begin{equation}\label{eq:symbol-ideal}
J_{A,L}=W\ini_v(I_A)+(e_1,\ldots,e_d).
\end{equation}
By \cite[Proposition~3.8 and equation~(3.2.2)]{SW08},
\begin{equation}\label{eq:sw-support}
V(J_{A,L})=\bigcup_{\tau\in\Phi_A^L}\Lambda_\tau.
\end{equation}
These are distinct irreducible subvarieties of dimension $n$. By \cite[Corollary~4.17]{SW08}, their union is $\Ch^L(M_A(\beta))$ for every $\beta$. 

Let $B=\C[b_1,\ldots,b_d]$ and consider
\begin{equation}\label{eq:relative-gkz}
\MM_B=D_B\big/\bigl(D_BI_A+\textstyle\sum_{i=1}^dD_B(E_i-b_i)\bigr).
\end{equation}
We assign filtration degree zero to its cyclic generator. The associated graded module is then cyclic over the commutative ring $W_B$. Every element of $\ini_vI_A$ annihilates this module, and
\[
\sigma_L(E_i-b_i)=e_i,
\]
since $\deg_L(E_i)=c>0$ whereas $\deg_L(b_i)=0$. Therefore
\begin{equation}\label{eq:relative-ann}
J_{A,L}W_B\subseteq\Ann_{W_B}(\gr_F\MM_B).
\end{equation}
Combining this containment with \eqref{eq:sw-support} gives
\begin{equation}\label{eq:gkz-product}
\Supp\gr_F\MM_B
\subseteq\bigcup_{\tau\in\Phi_A^L}\Lambda_\tau\times\Spec B.
\end{equation}

\begin{proof}[Proof of Theorem~\ref{thm:gkz}]
We first reduce to the case $\Z A=\Z^d$. If $\Z A$ is a proper sublattice of $\Z^d$, choose a basis of $\Z A$ and let $U\in\operatorname{GL}_d(\Q)$ be the corresponding change-of-basis matrix. Replacing $(A,\beta)$ by $(UA,U\beta)$ preserves both the toric ideal and the defining ideal of the hypergeometric module. It also preserves the $(A,L)$-umbrella and the corresponding characteristic cycle multiplicities. Since the induced change of parameters is an algebraic isomorphism, it suffices to consider the case $\Z A=\Z^d$.

By right exactness of tensor product, the fiber of $\MM_B$ in \eqref{eq:relative-gkz} at $b=\beta$ is $M_A(\beta)$. The relative module $\MM_B$ satisfies the product support condition \eqref{eq:gkz-product}, and the induced filtrations on its fibers are good by Lemma~\ref{lem:basechange}. By \eqref{eq:cc-length}, the multiplicity along $\Lambda_\tau$ is $\mu_{A,0}^{L,\tau}(\beta)$. Theorem~\ref{thm:general} therefore implies that $\{\beta\in\C^d:\mu_{A,0}^{L,\tau}(\beta)\ge r\}$ is Zariski closed for every integer $r$, as required.
\end{proof}

We obtain two immediate consequences.

\begin{corollary}\label{cor:gkz-curve}
Consider a morphism from an irreducible algebraic curve to $\C^d$, and let $R$ be the local ring at a closed point of its normalization, with fraction field $K$ and residue field $k=\C$. Write $\beta_0$ and $\beta_\eta$ for the special and generic parameters, respectively. For $\MM_R=\MM_B\otimes_BR$, we have
\begin{equation}\label{eq:gkz-tor-difference}
\mu_{A,0}^{L,\tau}(\beta_0)-\mu_{A,0}^{L,\tau}(\beta_\eta)
=\mu^{L,\Lambda_\tau}\bigl(\Tor_1^R(\MM_R,k)\bigr)\ge0.
\end{equation}
Here the generic multiplicity is computed over $K$, and $\mu^{L,\Lambda_\tau}$ denotes the coefficient at $\Lambda_\tau$.
\end{corollary}

\begin{proof}
Apply Lemma~\ref{lem:basechange} to \eqref{eq:gkz-product}, and then take the coefficient of $\Lambda_\tau$ in Theorem~\ref{thm:curve}.
\end{proof}

\begin{corollary}\label{cor:gkz-loci}
For fixed $A$ and $L$, the vector $(\mu_{A,0}^{L,\tau}(\beta))_{\tau\in\Phi_A^L}$ takes only finitely many values. Let $\mu_A^{L,\tau}$ denote the generic multiplicity along $\Lambda_\tau$. Then
\[
\mathcal E_A^{L,\tau}=\{\beta:\mu_{A,0}^{L,\tau}(\beta)>\mu_A^{L,\tau}\}
\]
is a proper Zariski closed subset of $\C^d$.
\end{corollary}

\begin{proof}
Finiteness follows from Corollary~\ref{cor:finite-strata}. The generic value is attained on a nonempty open set, so $\mathcal E_A^{L,\tau}$ is proper. By Theorem~\ref{thm:gkz}, it is Zariski closed, since it is the upper level set with threshold $\mu_A^{L,\tau}+1$.
\end{proof}

For $L=F$ and $\tau=\varnothing$, Theorem~\ref{thm:gkz} recovers the upper semicontinuity of the holonomic rank of $A$-hypergeometric systems established in \cite[Theorem~2.6]{MMW05}. For arbitrary $L$, the generic coefficients and their lower-bound property are given by \cite[Theorems~4.21 and~4.28]{SW08}. Formula~\eqref{eq:gkz-tor-difference} also compares the multiplicities at the generic point of any parameter curve with those at its specializations, even when the curve is contained in an exceptional locus.

The formulas in \cite{Ber11,BFF20} express the multiplicities in terms of ranking toric modules and Euler--Koszul homology. The individual terms in the resulting alternating sums need not be nonnegative. By \eqref{eq:gkz-tor-difference}, the change in each multiplicity under parameter specialization is the coefficient of an effective cycle.

We recall the Euler--Koszul construction for use in Section~\ref{sec:ext}. As before, let $P=\C[\partial_1,\ldots,\partial_n]$ with the grading $\deg\partial_j=a_j$, and give $D$ the compatible grading $\deg x_j=-a_j$.


\begin{definition}[Euler--Koszul homology]\label{def:euler-koszul}
For a graded $P$-module $N$ and a parameter $\beta\in\C^d$, define commuting $D$-linear endomorphisms of $D\otimes_PN$ by specifying their action on homogeneous elements $m$:
\[
(E_i-\beta_i)\circ m=(E_i-\beta_i+\deg_i m)m.
\]
The \emph{Euler--Koszul complex} of $N$ at $\beta$ is the Koszul complex of these endomorphisms. We denote its $j$th homology by $H_j(N,\beta)$.
\end{definition}

For $N=S_A$, the zeroth homology is $H_0(S_A,\beta)=M_A(\beta)$. The other modules considered below admit filtrations by semigroup rings of faces of $A$.

\begin{definition}\label{def:toric-module}
A \emph{finite toric module} is a graded $P$-module admitting a finite filtration whose factors are graded translates of the semigroup rings of faces of $A$.
\end{definition}


If $N$ is a full-dimensional Cohen--Macaulay toric module, then $H_i(N,\beta)=0$ for all $i>0$ by \cite[Theorem~6.6]{MMW05}. Since its Euler characteristic cycle is independent of $\beta$ by \cite[Theorem~4.11]{SW08}, the characteristic cycle of $H_0(N,\beta)$ is also independent of $\beta$. This homological framework extends the earlier explicit computations for monomial curves and simplicial configurations \cite{CDD99,Sai02,MM06,Oku06}. The behavior of Euler--Koszul complexes under toric localization and their relation to twisted Gauss--Manin systems is studied in \cite{SW09}.

For a graded module $N$ and a multidegree $b$, we use the shift convention $N(b)_v=N_{v-b}$. In particular, $\C(b)$ is concentrated in degree $b$.

\begin{example}\label{ex:rankjump-seed}
Let
\[
T_0=\begin{pmatrix}1&1&1&1\\0&1&3&4\end{pmatrix},
\quad b=(1,2),
\]
and write $\beta_T\in\C^2$ for the parameter of the hypergeometric system associated with $T_0$, and let $\widetilde S_{T_0}$ denote the normalization of $S_{T_0}$.
This is the classical rank-jump configuration of Sturmfels--Takayama \cite{ST98}, also studied in \cite{MMW05,BFF20}. Its saturated semigroup consists of $(k,j)$ with $k\ge0$ and $0\le j\le4k$. The only missing degree is $b$: the sums of two elements of $\{0,1,3,4\}$ cover $\{0,\ldots,8\}$, and adding zero or four establishes the claim at every higher level. Hence
\begin{equation}\label{eq:normalization-T}
\widetilde S_{T_0}/S_{T_0}\simeq\C(b),
\end{equation}
where every positive-degree semigroup generator annihilates $\C(b)$. The Euler--Koszul homology of $\C(b)$ in degree $j$ is $\C[x]\otimes\bigwedge^j\C^2$ at $\beta_T=b$ and vanishes at all other parameters. The normalization is Cohen--Macaulay by \cite{Hoc72}. Its zero-section coefficient is four for every projective weight, since the columns of $T_0$ are homogeneous and the corresponding triangles in the umbrella have total normalized area four. 
The long exact sequence in Euler--Koszul homology therefore gives
\begin{equation}\label{eq:muT}
\mu_{T_0,0}^{L,\varnothing}(\beta_T)=4+\one_{\{b\}}(\beta_T),
\end{equation}
where $\one_E$ is the indicator function of a set $E$.
The correction at $b$ is $2-1=1$.
\end{example}

\begin{example}\label{ex:nonconvex-short}
For a nonconvex example with a parameter jump, let
\[
B_0=\begin{pmatrix}2&0&1\\0&3&1\end{pmatrix},
\quad A_\dagger=B_0\oplus T_0.
\]
The toric ring of $B_0$ is the hypersurface defined by $\partial_1^3\partial_2^2-\partial_3^6$. Its zero-section coefficient is independent of the parameter. Set $\delta(v)=6v_3-3v_1-2v_2$. The coefficient is five for $\delta(v)\le0$ and six for $\delta(v)>0$. In the first open region, the two umbrella edges have normalized areas $2$ and $3$, while the other open region gives an area of six. On the wall, the formula of \cite[Theorem~4.21]{SW08} subtracts the area of $\conv((2,0),(0,3),(1,1))$ from that of $\conv(0,(2,0),(0,3),(1,1))$, giving $6-1=5$.

Take $v=(2,3,1,1,1,1,1)$ and $u=4\mathbf1-v$. In the $B_0$ face, the unscaled triangles contain $(2,0)$ and $(0,3)$ but not their midpoint, so this weight is nonconvex in the sense of \cite[Section~6]{BFF20}. By the external-product presentation and the tensor-product filtration, we obtain
\[
\mu_{A_\dagger,0}^{L,\varnothing}(\beta)
=5\bigl(4+\one_{\{\beta_3=1,\,\beta_4=2\}}(\beta)\bigr).
\]
Thus the multiplicity increases from $20$ to $25$ on the indicated parameter plane. By Corollary~\ref{cor:gkz-curve}, this increase equals the coefficient of the zero section in the characteristic cycle of the Tor module. Since the column lattice of $B_0$ is $\Z^2$, the normalized areas here are computed in that lattice.
\end{example}


\section{Gevrey irregularity under specialization}\label{sec:gevrey}

We now apply our results on the specialization of characteristic cycles to Gevrey irregularity. Together with the Gevrey index formula, they yield upper semicontinuity of generic irregularity and its graded dimensions for $A$-hypergeometric systems.


\subsection{The Gevrey index formula}

We continue with $Y=\{x_n=0\}$ and work on the corresponding analytic spaces. We recall the Gevrey sheaves that define the invariants in \eqref{eq:irregularity}.

\begin{definition}[Gevrey series and irregularity]\label{def:gevrey-irregularity}
Let $s\ge1$ and let $\sum_{m\ge0}f_m(x')x_n^m$ be a formal series with holomorphic coefficients on $Y$. It has \emph{Gevrey order at most $s$} if the series $\sum_{m\ge0}f_m(x')x_n^m/(m!)^{s-1}$ converges locally. We denote the sheaf of such series by $\OO_{\widehat{X|Y}}(s)$. When $s=1$, this is the sheaf of convergent germs along $Y$.

For a holonomic $D$-module $N$ and $s>1$, define the \emph{irregularity complex of order $s$} and the generic rank of its zeroth cohomology by
\[
I_s(N)=\RH_{\DD_{X^{\mathrm{an}}}}\bigl(N^{\mathrm{an}},
\OO_{\widehat{X|Y}}(s)/\OO_{X|Y}\bigr),
\quad d_s(N)=\rank_Y\mathcal H^0 I_s(N).
\]
The rank is computed at a general point of $Y$.
\end{definition}

These Gevrey sheaves appear in \cite{Meb90,LM99} and in explicit constructions for hypergeometric systems \cite{FF10,FC11}. We use their sheaf-theoretic description; comparisons with explicit bases of Nilsson series can be found in \cite{DMM12,FF20}. To describe the graded dimensions, we also consider quotients between two Gevrey levels.

\begin{definition}\label{def:gevrey-graded}
Let $N$ be holonomic, and let $1<a<b$ be non-slope levels. We denote by $e_{a,b}(N)$ the generic dimension along $Y$ of
\[
\mathcal Hom_{\DD_{X^{\mathrm{an}}}}\bigl(N^{\mathrm{an}},
\OO_{\widehat{X|Y}}(b)/\OO_{\widehat{X|Y}}(a)\bigr).
\]
For $\rho>1$, choose non-slope levels $1<a<\rho<b$ such that $[a,b]$ contains no slope of $N$ other than possibly $\rho$. We write $e_\rho(N)=e_{a,b}(N)$ for the generic dimension of the graded piece of the Gevrey filtration at $\rho$.
\end{definition}

We next express these dimensions in terms of characteristic cycles. For rational $r>1$, set $L(r)=F+(r-1)V_Y$. Thus $u_i=0$, $v_i=1$ for $i<n$, and $u_n=1-r$, $v_n=r$. Let $Z=T_X^*X$ be the zero section and $C_Y=T_Y^*X$ the conormal bundle of $Y$. We take the coefficient of any component absent from a characteristic cycle to be zero. Define
\[
\phi_r(N)=m_Z^r(N)-m_Y^r(N),
\]
where $m_Z^r(N)$ and $m_Y^r(N)$ are the coefficients of $Z$ and $C_Y$, respectively, in $\CC^{L(r)}(N)$.

Recall that $W=\C[x,\xi]$ is the associated graded ring of $D$. A subvariety of $\Spec W$ is called $(F,V_Y)$-bihomogeneous if it is invariant under the two grading actions defined by $F$ and $V_Y$.

\begin{lemma}\label{lem:index}
Let $N$ be holonomic. Assume that its algebraic slopes along $Y$ lie in a finite set $\Sigma\subset\Q_{>1}$ and that its $L(r)$-characteristic varieties are $(F,V_Y)$-bihomogeneous for $r\notin\Sigma$. Choose a rational $r_0>1$ smaller than every element of $\Sigma$, and a rational $r_+>s$ such that no element of $\Sigma$ lies in $(s,r_+]$. Then
\begin{equation}\label{eq:G}
d_s(N)=\phi_{r_+}(N)-\phi_{r_0}(N).
\end{equation}
For non-slope levels $1<a<b$, the dimensions defined in Definition~\ref{def:gevrey-graded} satisfy
\begin{equation}\label{eq:grade}
e_{a,b}(N)=d_b(N)-d_a(N)\geq0.
\end{equation}
\end{lemma}

\begin{proof}
We apply the Gevrey index theorem using the following conventions. Laurent and Mebkhout \cite[Theorem~2.3.1 and \S2.4]{LM99}, using the positivity theorem of Mebkhout \cite{Meb90}, prove that the irregularity complexes, together with their Gevrey maps, form an increasing filtration in the abelian category of perverse sheaves, and that the graded objects are also perverse. Under the usual middle-perversity convention of \cite{BBD82,KS90}, in which a local system on the smooth $(n-1)$-fold $Y$ is shifted by $n-1$, the complexes here are $I_s(N)[n-1]$. The unshifted convention in \eqref{eq:irregularity} is the one used for the degree-zero ranks in \cite[\S7]{BFF20}. On a dense open subset of $Y$, each unshifted irregularity complex and each unshifted graded complex has cohomology only in degree zero. Consequently, the generic Euler characteristic of each such complex equals the nonnegative dimension of its degree-zero cohomology.

The index formula uses the conventions for characteristic cycles and Euler obstructions in microlocal index theory \cite{Kas85,KS90}. We now specify the sign convention for the Gevrey formula in \cite{LM99}. In coordinates $(x',y;\xi',\eta)$, the symplectic identification with the cotangent bundle of $\Lambda=T_Y^*X$ is
\[
(x',y;\xi',\eta)\longmapsto(x',\eta;\xi',-y).
\]
This is the identification used for $(F,V_Y)$ in \cite[Remark~3.3]{SW08}. We distinguish two numerical conventions for slopes. On p.~57 of \cite{LM99}, a rational level $r=p/q>1$ defines the filtration $qF+(p-q)V_Y=qL(r)$. The weight of $y^l\partial_y^k\partial_{x'}^\alpha$ under this filtration is $p(k-l)+q(l+|\alpha|)$. Thus the Gevrey level in this paper is exactly the parameter in that filtration. The microlocal type in \cite[Remark~3.3]{SW08} for $aF+bV_Y$ is $1+a/b$; for our $L(r)$, this label is $r/(r-1)$, whereas the Gevrey order is $r$. The zero section of $T^*X$ becomes the conormal of the zero section $Y\subset\Lambda$, while $C_Y$ becomes the zero section of $T^*\Lambda$. The Euler-obstruction signs of these two cycles at a general point of $Y\subset\Lambda$ are $-1$ and $+1$, respectively. The index formula for formal completion in \cite[p.~64]{LM99} includes an additional minus sign. Consequently, the generic index for solutions with values in a Gevrey completion at a non-slope level is $m_Z^r-m_Y^r$.

At a general point of $Y$, only these two components contribute. At a bihomogeneous level, the cycles are conormal cycles. If the irreducible base of another conormal in $T^*\Lambda$ contained the zero section $Y\subset\Lambda$, its dimension would force it to be $Y$ or $\Lambda$. All other bases meet $Y$ in proper closed subsets and can therefore be avoided when computing the generic index.

The index at the convergent level is $\phi_{r_0}$. At Gevrey level $s$, the index is $\phi_{r_+}$, with the one-sided convention following from the open and closed Gevrey levels in \cite[Theorem~2.3.1]{LM99}. In particular, a slope equal to $s$ is included. The same convention is used in \cite[equation~(7-1)]{BFF20}. The short exact sequence defining the quotient by convergent series gives \eqref{eq:G} as the difference of these two indices. Since the cohomology is generically concentrated in degree zero, this difference equals $d_s(N)$.

To show \eqref{eq:grade}, notice that the exact sequence of coefficient sheaves gives the distinguished triangle
\[
I_a(N)\longrightarrow I_b(N)\longrightarrow
\RH_{\DD_{X^{\mathrm{an}}}}\bigl(N^{\mathrm{an}},\OO_{\widehat{X|Y}}(b)/\OO_{\widehat{X|Y}}(a)\bigr)\xrightarrow{\,+1\,}.
\]
By the theorem on the Gevrey filtration cited above, the first arrow becomes a monomorphism in the perverse category after shifting by $n-1$. Its cone is the perverse quotient. On a common dense open subset of $Y$, this gives an exact sequence of local systems, unshifted to degree zero. 
The degree-zero cohomology of the third term is the ordinary $\mathcal Hom$ in Definition~\ref{def:gevrey-graded}. Taking ranks in this exact sequence thus gives \eqref{eq:grade}.
\end{proof}

The same exactness shows that $d_s$ is additive in short exact sequences of holonomic modules, taking into account that the solution functor reverses arrows. Modules whose irregularity vanishes along $Y$ form a Serre subcategory. These assertions hold at every real level: the finitely many slopes are rational by \cite[\S2.1 and Theorem~2.4.2]{LM99}, and the value is constant between consecutive slopes. At a slope, we use the value immediately to its right. By \eqref{eq:grade}, $e_\rho(N)$ is independent of the sufficiently close non-slope levels used in Definition~\ref{def:gevrey-graded}.


\subsection{Support, slopes, and specialization}

Let $M$ be a finite $D_B$-module. Assume that
\begin{equation}\label{eq:gevrey-support}
M\text{ satisfies \eqref{eq:product-support} for }F\text{ and every rational }L(r),\ r>1,
\end{equation}
where the fixed set of components may depend on the weight. Assume further that there is a finite set $\Sigma\subset\Q_{>1}$ such that
\begin{equation}\label{eq:gevrey-bihom}
\text{each prescribed component for }L(r)\text{ is }(F,V_Y)\text{-bihomogeneous if }r\notin\Sigma.
\end{equation}
Under the symplectic identification above, the latter condition means that each such component is invariant under both grading actions defined by $F$ and $V_Y$.

\begin{lemma}\label{lem:tor-control}
Under \eqref{eq:gevrey-support} and \eqref{eq:gevrey-bihom}, for any pullback to a DVR $R$ associated with a parameter curve, the module
\[
T=\Tor_1^R(M_R,k)=\ker(t:M_R\to M_R)
\]
is a finite holonomic $D$-module. All complex fibers and $T$ have analytic slopes along $Y$ in $\Sigma$, and satisfy Lemma~\ref{lem:index} with the same choices of levels.
\end{lemma}

\begin{proof}
Since $D_R$ is Noetherian, the kernel is finite over $D_R$. As it is annihilated by $t$, it is finite over $D_R/(t)\simeq D$. For each weight, equip the kernel with the induced filtration. Its associated graded module injects into that of the curve family, so its support is contained in the special slice of the prescribed product support. The support condition for the order filtration then shows that $T$ is holonomic.

By \cite[Theorem~3.2]{SW08}, every irreducible component of the characteristic variety of a finite $D$-module with respect to a projective filtration has dimension at least $n$. In a finite union of irreducible $n$-dimensional varieties, a closed subset with this property is a union of some of those varieties. Hence, for $r\notin\Sigma$, the characteristic variety of $T$ is bihomogeneous. The same argument applies to each fiber, using the surjection \eqref{eq:basechange-surj} for the induced filtration. Thus no algebraic slope occurs at such $r$. By \cite[Theorem~2.4.2]{LM99}, the analytic and algebraic slopes coincide, hence the analytic slopes of $T$ and all fibers lie in $\Sigma$.
\end{proof}

\begin{theorem}\label{thm:gevrey-general}
Under \eqref{eq:gevrey-support} and \eqref{eq:gevrey-bihom}, $d_s(M_b)$ and the generic dimension of each Gevrey graded piece are Zariski upper semicontinuous. More precisely, on every parameter curve,
\begin{align}
 d_s(M_0)-d_s(M_{\mathrm{gen}})&=d_s(T),\label{eq:Gtor}\\
 e_\rho(M_0)-e_\rho(M_{\mathrm{gen}})&=e_\rho(T).\label{eq:Etor}
\end{align}
Here $M_{\mathrm{gen}}$ denotes a general complex fiber of the parameter curve, and the right-hand sides are nonnegative.
\end{theorem}

\begin{proof}
Choose a rational representative in each component of $(1,\infty)\setminus\Sigma$. For a given $s$, choose the representative in its interval or immediately to its right, together with the representative in the interval adjacent to $1$. For a given $\rho>1$, choose representatives on both sides; equivalently, first take $a=\rho-\epsilon$ and $b=\rho+\epsilon$ with rational $\epsilon>0$ sufficiently small. These choices give the same graded dimension. The resulting list is finite and independent of the fiber and the parameter curve.

If $s$ lies below every member of $\Sigma$, the two cycle expressions agree and $d_s=0$. If $\rho\notin\Sigma$, choose both levels in the same interval; the graded object then vanishes. The same quotient handles the first slope and multiple changes occurring at a single slope. All facets that change at a repeated slope contribute to this one object. Since only finitely many modules enter each comparison, the dense open subsets of $Y$ on which their generic descriptions hold have a common dense open intersection.

By Lemmas~\ref{lem:index} and~\ref{lem:tor-control}, all the relevant invariants can be expressed as fixed linear combinations at these weights, both for complex fibers and for $T$. Applying Theorem~\ref{thm:curve} gives \eqref{eq:Gtor} and \eqref{eq:Etor}. These invariants are generic dimensions of an irregularity object or its perverse graded quotient, and are therefore nonnegative. By Proposition~\ref{prop:positive-functional}, the asserted upper semicontinuity follows.
\end{proof}


\subsection{Upper semicontinuity of Gevrey irregularity}

We now apply Theorem~\ref{thm:gevrey-general} to the relative $A$-hypergeometric family \eqref{eq:relative-gkz}.

\begin{proof}[Proof of Theorem~\ref{thm:gkz-app}]
By \eqref{eq:gkz-product}, the cyclic family \eqref{eq:relative-gkz} satisfies \eqref{eq:gevrey-support} for every projective weight. Along $L(r)$ there are finitely many umbrella walls. These walls can also be determined from the finitely many supporting hyperplanes of $\conv(0,a_1,\ldots,a_{n-1})$ that do not pass through the origin: possible changes occur when $a_n/r$ meets one of them. Degenerate cases can be read directly from the finite rational fan of initial ideals. The walls are rational. Away from these walls, the umbrella components are $(F,V_Y)$-bihomogeneous, as in the slope criterion of \cite[Corollary~4.18]{SW08} and \cite[Remark~7.1]{BFF20}. Thus \eqref{eq:gevrey-support} and \eqref{eq:gevrey-bihom} hold.

The right-hand convention ensures that the rational representatives compute the value at the specified Gevrey level even when that level is a slope. The support theorem applies equally to nonconvex weights and to cases in which several facets change at the same slope. By Theorem~\ref{thm:gevrey-general}, both assertions follow.
\end{proof}

For comparison, \cite{CG15,CFG21} realize Gevrey series as asymptotic expansions of integral solutions, while \cite{ET15} uses rapid decay homology in the confluent setting and \cite{CFKT15} studies summation through modified systems. Such descriptions provide information about individual solutions beyond their dimensions.

The following example shows that flatness alone does not control irregularity.

\begin{example}\label{ex:pressure}
Consider the family
\[
M=D_{\A^1}[t]/D_{\A^1}[t](x^2\partial-t).
\]
As a left $D$-module, $M$ is isomorphic to $D$, with the central parameter $t$ acting by right multiplication by $x^2\partial$. The highest-order symbol shows that no nonzero polynomial in $x^2\partial$ vanishes, and the Weyl algebra is a domain. Hence $M$ is torsion-free and therefore flat over $\C[t]$ and its localization at $(t)$.

For $t\ne0$, the connection away from $x=0$ has solution $\exp(-t/x)$ and slope $2$. Its generic irregularity along $x=0$ is $1$ for $s\geq2$ and $0$ for $s<2$. At $t=0$, the quotient $D/D(x^2\partial)$ is regular. Factoring $x^2\partial$ gives an extension with factors $D/D\partial$ and $D/Dx^2$, the latter supported at the point. The special fiber has zero irregularity at every level. For the order filtration, both fibers have characteristic cycle $[Z]+2[T_0^*\A^1]$.

For the filtrations considered in Theorem~\ref{thm:gevrey-general}, the relative principal symbol is
\[
\begin{cases}
x^2\xi,&1<r<2,\\
x^2\xi-t,&r=2,\\
t,&r>2.
\end{cases}
\]
At $r=2$, it defines curves that vary with the parameter, rather than a finite union of products of fixed curves with the parameter line. For $r>2$, its special support is the entire two-dimensional cotangent space, rather than a union of fixed one-dimensional components. Thus the support conditions for $L(r)$ fail, while the condition for the order filtration holds, in agreement with Corollary~\ref{cor:flatness}.
\end{example}

\begin{example}\label{ex:multiple-slopes}
We consider two normal configurations in which the relevant facets contribute at distinct slopes and at the same slope, respectively. Let
\[
A=(e_1,e_2,(4,0),(3,2),(0,4),a_6),\quad Y=\{x_6=0\}.
\]
The semigroup of $A$ is $\N^2$, so the relevant characteristic cycles and irregularities are constant in the parameter. The two facets of $\conv(0,A\setminus\{a_6\})$ that do not contain the origin have supporting forms $(1/4,1/8)$ and $(1/6,1/4)$, with normalized triangle volumes $8$ and $12$, respectively. If $a_6=(8,8)$, the slopes are $3$ and $10/3$. The corresponding values of $d_s$ are $0$, $8$, and $20$ for $1<s<3$, $3\leq s<10/3$, and $s\geq10/3$, respectively. Thus the two graded dimensions are $8$ and $12$. If $a_6=(9,6)$, both facets change at $3$, giving a single graded dimension of $20$. Since $a_6/3=(3,2)$, both facets adjacent to $(3,2)$ contribute to the jump at $s=3$. By the generic formula \cite[equation~(7-2)]{BFF20} and the Cohen--Macaulay property, these statements hold for all parameters.
\end{example}


\section{The exceptional arrangement and Gevrey jumps}\label{sec:ext}

In this section, we compare the jump loci of generic Gevrey irregularity and its graded dimensions with the exceptional arrangement defined by local cohomology. We show that these loci are contained in this arrangement and construct non-Cohen--Macaulay examples for which the inclusion is strict.


\subsection{Local cohomology and partial saturation}

We use the positive grading $\deg\partial_j=a_j$ and $\deg x_j=-a_j$ fixed in Section~\ref{sec:gkz}, together with the translation convention of \cite[\S1]{BFF20}:
\[
N(b)_v=N_{v-b},\quad \deg N(b)=b+\deg N,
\quad H_i(N(b),\beta)=H_i(N,\beta-b)(b).
\]
Here $H_i$ denotes Euler--Koszul homology. We also write $\mathfrak m$ for the image of the ideal $(\partial_1,\ldots,\partial_n)\subset P$ in $S_A$. Set $\varepsilon_A=\sum_j a_j$. We define the exceptional arrangement of a toric module as follows.

\begin{definition}\label{def:exceptional-arrangement}
The \emph{exceptional arrangement} of a finite full-dimensional toric module $N$ is
\begin{equation}\label{eq:exceptional}
\mathcal Z(N)=\overline{\bigcup_{q<d}\deg H^q_{\mathfrak m}(N)}
=-\qdeg\bigoplus_{q<d}\Ext_P^{n-q}(N,\omega_P),
\quad \omega_P=P(\varepsilon_A).
\end{equation}
Here $\qdeg$ denotes the Zariski closure of the set of degrees of nonzero homogeneous elements, and the overline on the left has the same meaning.
\end{definition}

The second equality in \eqref{eq:exceptional} follows from graded local duality \cite[Chapter~3]{BH98}; see also \cite[Chapter~13]{MS05} for multigraded local cohomology. The grading on the canonical module is given in \eqref{eq:exceptional}. With the positive grading fixed above, \cite[Theorem~6.6]{MMW05} implies that the higher Euler--Koszul homology of $N$ vanishes outside $\mathcal Z(N)$.

For irregularity along $Y$, we use the partial saturation defined in \cite[Notation~5.4]{BFF20}.

\begin{definition}\label{def:partial-saturation}
The \emph{partial saturation} of $S_A$ along $a_n$ is the ring
\begin{equation}\label{eq:S-n}
N=S_A^{\{n\}}=\C[(\N A+\Z a_n)\cap\R_{\geq0}A].
\end{equation}
\end{definition}

This ring is independent of the Gevrey level. It lies between $S_A$ and its finite normalization and is a submodule of $S_A[\partial_n^{-1}]$. Hence it is a finite toric $P$-module. For the theory of toric modules and localization, see \cite{MMW05,SW09}. 
We write $d_s(A,\beta)=d_s(M_A(\beta))$, and denote by $d_s(A)$ and $e_\rho(A)$ the generic values as the parameter varies. Their existence follows from Theorem~\ref{thm:gkz-app} and the finite-stratification argument in its proof.

\begin{proposition}\label{prop:ext-inclusion}
For every $s>1$ and every $\rho>1$,
\[
\{\beta:d_s(M_A(\beta))>d_s(A)\}\subseteq\mathcal Z(S_A^{\{n\}}).
\]
The same inclusion holds for the locus where a graded dimension exceeds its generic value. If $S_A^{\{n\}}$ is Cohen--Macaulay, all these generic dimensions are constant in $\beta$.
\end{proposition}

\begin{proof}
Write $N=S_A^{\{n\}}$ and $Q=N/S_A$. Since $N$ is finite and embeds in $S_A[\partial_n^{-1}]$, some power of $\partial_n$ annihilates $Q$. A finite toric filtration of $Q$ has factors that are translates of $S_G$ with $a_n\notin G$. Indeed, if $a_n\in G$, multiplication by $\partial_n$ is injective on $S_G$, contradicting its nilpotence on the factor.

For each such factor, the Euler--Koszul complex is the external product of a complex in the variables indexed by $G$ with $\OO$ in the inactive variables, including $x_n$. This follows directly from $D\otimes_P S_G$: after the degree correction in the Euler endomorphisms, the inactive variables have no Euler action. The resulting homology is noncharacteristic for $Y$ and has zero irregularity along $Y$ at every level. By exactness of the irregularity functor in the perverse category, this vanishing is preserved under subquotients and extensions. Induction on the toric filtration, together with the long exact sequences in Euler--Koszul homology, shows that every $H_i(Q,\beta)$ has zero irregularity. Consequently, $H_0(S_A,\beta)$ and $H_0(N,\beta)$ have the same irregularity objects, and therefore the same generic graded dimensions, for every $\beta$.

Outside $\mathcal Z(N)$, the higher Euler--Koszul homology of $N$ vanishes. By \cite[Theorem~4.11]{SW08}, its Euler characteristic cycle is independent of $\beta$. This cycle equals that of $S_A$: the quotient $Q$ has a filtration by modules associated with proper faces, and each such factor has an exterior Euler factor $(1-1)^{\operatorname{codim}G}=0$. Thus, outside $\mathcal Z(N)$, $H_0(N,\beta)$ has the generic characteristic coefficients of $S_A$ at every weight used in Lemma~\ref{lem:index}. This proves the inclusions and the Cohen--Macaulay implication.
\end{proof}

\begin{remark}\label{rem:shift}
The Ext expression in \cite[Conjecture~7.10]{BFF20} uses the shift $-\varepsilon_A$. With the convention $N(b)_v=N_{v-b}$, graded local duality takes the form \eqref{eq:exceptional}. For example, $\Ext_P^n(\C(b),P)$ has a generator of degree $-b-\varepsilon_A$. We use the canonical-module expression throughout. The counterexample below has an empty Gevrey jump locus and a nonempty exceptional arrangement, so the failure of equality persists under either convention.
\end{remark}


\subsection{A non-Cohen--Macaulay family without Gevrey jumps}

\begin{theorem}\label{thm:counterexample}
For every integer $K\geq4$, let
\begin{equation}\label{eq:AK}
A_K=\begin{pmatrix}
1&1&1&1&3&3&2&0\\
0&1&3&4&0&12&4&0\\
0&0&0&0&0&0&1&K
\end{pmatrix},\quad Y=\{x_8=0\}.
\end{equation}
Then $\Z A_K=\Z^3$, its cone is pointed, and $S_{A_K}^{\{8\}}=S_{A_K}$ has depth $2$ and dimension $3$. Nevertheless, for every $\beta\in\C^3$,
\begin{equation}\label{eq:constant-irreg}
d_s(A_K,\beta)=
\begin{cases}
0,&1<s<K/3,\\
36,&s\geq K/3.
\end{cases}
\end{equation}
The only slope along $Y$ is $K/3>1$. The canonical exceptional arrangement is the following nonempty union of lines:
\begin{equation}\label{eq:ZK}
\mathcal Z(S_{A_K})=
\bigcup_{j=0}^{K-1}\{(1+2j,2+4j,z):z\in\C\}.
\end{equation}
Thus the reverse inclusion in Proposition~\ref{prop:ext-inclusion}, the equality with the full exceptional arrangement defined by Ext, and the converse to the Cohen--Macaulay implication all fail, even at levels above the unique slope.
\end{theorem}

\begin{proof}
Recall $T_0$ and its hole $b$ from Example~\ref{ex:rankjump-seed}. The fifth and sixth columns in \eqref{eq:AK} are redundant generators $3a_1$ and $3a_4$. In the semigroup algebra, introduce $z$ of degree $(0,0,1)$, let $m\in S_{T_0}$ be the monomial of degree $(2,4)$, and put $q=mz$ and $w=z^K$. Then
\begin{equation}\label{eq:SA-presentation}
S:=S_{A_K}=S_{T_0}[w,q]/(q^K-m^K w).
\end{equation}
Division by the monic relation shows that $1,q,\ldots,q^{K-1}$ span $S$ over $S_{T_0}[w]$. These elements have distinct $z$-exponents modulo $K$ and are therefore linearly independent in the semigroup algebra. Equivalently, the degree set consists exactly of
\begin{equation}\label{eq:degrees-S}
(u+j(2,4),\ j+K\ell),\quad u\in\N T_0,\quad 0\leq j<K,\quad\ell\geq0.
\end{equation}
Saturating by the last column allows $\ell\in\Z$. Intersecting with the cone forces $j+K\ell\geq0$, hence $\ell\geq0$. Thus $S^{\{8\}}=S$.

Consider the finite toric module
\[
N=\widetilde S_{T_0}[w,q]/(q^K-m^K w).
\]
It is free with basis $1,q,\ldots,q^{K-1}$ over the Cohen--Macaulay ring $\widetilde S_{T_0}[w]$ and is therefore Cohen--Macaulay of dimension $3$. The quotient is
\begin{equation}\label{eq:QK}
Q=N/S\simeq \bigl(\C[q,w]/(q^K)\bigr)(b,0),
\quad \deg q=(2,4,1),\quad\deg w=(0,0,K).
\end{equation}
Every monomial of positive degree in $S_{T_0}$ annihilates its unique normalization hole. Thus the first six differential variables act by zero in \eqref{eq:QK}, while multiplication by $q$ advances the residue $j$ until the last step, where multiplication by $m^K w$ kills the hole. These actions determine the displayed module isomorphism.

The quotient $Q$ has depth and dimension $1$: it is free over $\C[w]$ of rank $K$, and $q$ is nilpotent. By the depth lemma \cite[Chapter~1]{BH98}, applied to $0\to S\to N\to Q\to0$, we obtain $\depth S=2$. Since $N$ has no local cohomology in degrees below $3$, the same exact sequence gives
\[
H^0_{\mathfrak m}(S)=H^1_{\mathfrak m}(S)=0,
\quad H^2_{\mathfrak m}(S)\simeq H^1_{\mathfrak m}(Q).
\]
Since the radical of the positive ideal on $Q$ is $(q,w)$ and $q$ is nilpotent, we have $H^1_{\mathfrak m}(Q)=Q_w/Q$. Its degrees are
\[
b_j-\ell a_8,\quad b_j=(1+2j,2+4j,j),\quad 0\leq j<K,\quad\ell\geq1.
\]
The Zariski closure of these degrees is the arrangement in \eqref{eq:ZK}. We can also recover this arrangement using Ext. Put $b_0=(1,2,0)$, $J=(\partial_1,\ldots,\partial_6,\partial_7^K)\subset P$, and $\Delta=\sum_{i=1}^6a_i+Ka_7$. Then $Q=(P/J)(b_0)$ is a complete intersection module. Since $N$ is Cohen--Macaulay of dimension $3$ over the regular ring $P$ of dimension $8$, the Auslander--Buchsbaum formula \cite[Chapter~1]{BH98} gives projective dimension $5$. The long exact sequence in Ext gives
\[
\Ext_P^6(S,P)\simeq\Ext_P^7(Q,P)
\simeq(P/J)(-b_0-\Delta),
\]
and the other terms $\Ext_P^{8-q}(S,P)$ for $q<3$ vanish. This follows from the graded Koszul resolution of the seven displayed generators of $J$, with the grading of the canonical module as in \cite[Chapter~3]{BH98}. Using $\omega_P=P(\varepsilon_A)$ and negating the resulting quasidegrees gives precisely \eqref{eq:ZK}, in agreement with the calculation using local cohomology.

We next compute the Euler--Koszul homology of $Q$. The filtration of $Q$ by powers of $q$ has factors of the form $\C[w](b_j)$. Their Euler--Koszul homology vanishes unless $\beta\in b_j+\C a_8$. On this line, the homology in degree $i$ is
\begin{equation}\label{eq:Qhom}
\OO_{\A^7}\boxtimes
D_8/D_8\bigl(Kx_8\partial_8-(\beta_3-j)\bigr)
\otimes_\C\bigwedge^i\C^2,
\end{equation}
where $D_8=\C[x_8]\langle\partial_8\rangle$ be the Weyl algebra in the variable $x_8$.
The two Euler directions that act trivially give the exterior factor. The lines have distinct first two coordinates. Hence, for any parameter, at most one filtration factor has nonzero homology; all other factors are acyclic. The long exact sequences then give \eqref{eq:Qhom} for $H_i(Q,\beta)$ on the relevant line, and zero off their union. The Kummer module in \eqref{eq:Qhom} is regular along $x_8=0$ for every parameter, including integral ones. Its meromorphic connection has a logarithmic pole, as in the theory of regular singular connections \cite[Chapter~II]{Del70}; point-supported subquotients at integral exponents remain regular in the holonomic category \cite[Chapter~5]{HTT08}. Therefore $I_s(H_i(Q,\beta))=0$ for all $i,s,\beta$.

Since $N$ is Cohen--Macaulay, its higher Euler--Koszul homology modules vanish for all $\beta$ by \cite[Theorem~6.6]{MMW05}. Consequently, $\CC^L(H_0(N,\beta))$ is the Euler characteristic cycle of the Euler--Koszul complex, which is independent of $\beta$ by \cite[Theorem~4.11]{SW08}. This holds for every projective weight $L$, including $L(r)$. The long exact sequence in Euler--Koszul homology for $0\to S\to N\to Q\to0$, together with the regularity just proved, implies
\[
I_s(H_0(S,\beta))\simeq I_s(H_0(N,\beta))
\]
for every parameter. Since $N/S$ has smaller dimension, the constant Euler cycle of $N$ is the generic cycle of $S$. Thus it remains to compute the generic irregularity.

The columns of $T_0$ generate $\Z^2$, and $a_7$ differs from an element of their semigroup by $(0,0,1)$; hence $\Z A_K=\Z^3$. The cone is the product of the pointed cone of $T_0$ with the positive third ray. After deleting $a_8$, the polytope is the simplex with vertices $0,p_*,q_*,c_*$, where
\[
p_*=(3,0,0),\quad q_*=(3,12,0),\quad c_*=(2,4,1).
\]
Its only facet not containing the origin is supported by the hyperplane $\ell(x,y,z)=(x+z)/3=1$. The moving point $a_8/r$ meets this hyperplane exactly at $r=K/3$. Before this value, $c_*$ lies inside $\conv(0,p_*,q_*,a_8/r)$, since its barycentric coefficients are $1/3,1/3,r/K$ and their sum is less than $1$. After this value, the new facet not containing $a_8$ is $\{p_*,q_*,c_*\}$. The generic Gevrey formula \cite[equation~(7-2)]{BFF20} therefore gives its normalized volume
\[
|\det(p_*,q_*,c_*)|=36.
\]
At the wall, this contribution is counted at $s=K/3$, according to the convention used above. By the slope criterion of \cite[Corollary~4.18]{SW08}, $K/3$ is the only slope. This proves \eqref{eq:constant-irreg}.
\end{proof}

For a face $G$ of $A_K$ and a parameter $\beta$, put $E_G^\beta=((\beta+\C G)\cap\Z^3)\setminus(\N A_K+\Z G)$. Following \cite[\S4]{BFF20}, the constituent cosets $b+\Z G$ are called ranking lattices. A ranking pair $(G,b)$ is maximal if its coset is not properly contained in a ranking lattice associated with another pair. We call a face occurring in such a maximal pair a maximal ranking face.
The family in Theorem~\ref{thm:counterexample} also illustrates a point concerning \cite[Corollary~7.7]{BFF20}. For parameters on the exceptional lines, the face $G=\{a_8\}$ occurs among the maximal ranking faces, but the corresponding face system is regular and hence $d_s(G)=0$. Thus the presence of such a face does not by itself imply a strict Gevrey jump.

The ordinary rank also detects the contribution of regular modules. The polytope of $A_K$ has normalized volume $|\det(p_*,q_*,a_8)|=36K$, and the homology modules in \eqref{eq:Qhom} have ranks $1,2,1$ on each active line. Consequently,
\[
\rank M_{A_K}(\beta)=36K+\one_{\mathcal Z(S_{A_K})}(\beta),
\]
although the irregularity in \eqref{eq:constant-irreg} is independent of $\beta$. More precisely, on an active line, the difference between the characteristic cycle of $H_0(S,\beta)$ and the constant cycle of $H_0(N,\beta)$ is the cycle of \eqref{eq:Qhom} for $i=0$, since the exterior coefficients give $2-1=1$. The coefficients of the zero section and of $C_Y$ in this additional cycle are both $1$ for every $L(r)$, so their contributions to $\phi_r$ cancel at every level. Thus this additional contribution to the characteristic cycle is nonzero but makes no contribution to irregularity.

\begin{proposition}\label{prop:CM-cover}
Let $N$ be a finite full-dimensional Cohen--Macaulay toric module containing $S_A$, and assume that $Q=N/S_A$ has smaller dimension. Then, for every level and parameter,
\begin{equation}\label{eq:CM-defect}
d_s(A,\beta)-d_s(A)
=d_s(H_1(Q,\beta))-d_s(H_0(Q,\beta)).
\end{equation}
The same formula holds for the generic dimension of an individual graded piece. For a specialization along a parameter curve, however, the difference between the special and generic fiber values is given by the single nonnegative term in \eqref{eq:Gtor} or \eqref{eq:Etor}.
\end{proposition}

\begin{proof}
The higher Euler--Koszul homology of $N$ vanishes. Split
\[
0\to H_1(Q,\beta)\to H_0(S_A,\beta)\to H_0(N,\beta)\to H_0(Q,\beta)\to0
\]
into two short exact sequences and apply the additivity of generic irregularity. As in the preceding proof, the constant Euler cycle of $N$ is the generic cycle of $S_A$. Since graded dimensions are differences of two additive generic irregularities, the same identity holds for them. The assertion for parameter curves follows from Theorem~\ref{thm:gevrey-general}.
\end{proof}

In the family of Theorem~\ref{thm:counterexample}, both terms on the right-hand side of \eqref{eq:CM-defect} vanish, although the Euler--Koszul homology of $Q$ is nonzero. More generally, \eqref{eq:CM-defect} detects only the part of that homology measured by irregularity along $Y$. The exceptional arrangement defined by local cohomology records the full homological obstruction, whereas the Gevrey specialization formulas record its irregular contribution.

\bigskip 

\noindent\textbf{AI Declaration.}
The author used ChatGPT 6-Astra to explore proof strategies, refine mathematical arguments, and assist with revisions of the manuscript he wrote. The author takes full responsibility for the mathematical content and the final text.

\bigskip 


\end{document}